\documentclass[11pt,reqno]{amsart}
\usepackage[utf8]{inputenc}
\usepackage{float}
\usepackage{amsmath,hyperref,enumitem}
\usepackage{graphicx}
\usepackage{latexsym}
\usepackage{amsfonts}
\usepackage{amssymb,stmaryrd}
\usepackage{color}
\usepackage[capitalise,noabbrev]{cleveref}
\newcommand{\version}{Version of \today}
\theoremstyle{plain}
\newtheorem{theorem}{Theorem}
\newtheorem{corollary}[theorem]{Corollary}
\newtheorem{lemma}[theorem]{Lemma}
\newtheorem{proposition}[theorem]{Proposition}
\theoremstyle{definition}

\newtheorem*{remark*}{Remark}

\usepackage{multicol,ulem,tikz}
\usepackage{tikz,wasysym}
\usetikzlibrary{mindmap}
\usetikzlibrary{shapes,backgrounds}
\usepackage{tikz-3dplot}
\usepackage{tabulary}
\usetikzlibrary{shapes,calc,positioning}
\usetikzlibrary{calc,intersections}
\usepackage{tikz}
\usetikzlibrary{3d,calc,fadings,decorations.pathreplacing}

\tikzset{%
  >=latex, 
  inner sep=0pt,%
  outer sep=2pt,%
  mark coordinate/.style={inner sep=0pt,outer sep=0pt,minimum size=3pt,
    fill=black,circle}%
}

\usepackage{amssymb,amsfonts,amsmath}

\renewcommand {\epsilon}{\varepsilon}
\renewcommand {\le}{\leqslant}
\renewcommand {\ge}{\geqslant}
\renewcommand {\leq}{\leqslant}
\renewcommand {\geq}{\geqslant}

\newcommand{\pr}{\mathbf P}
\newcommand{\e}{\mathbf E}

\DeclareMathOperator{\dist}{dist}

\DeclareMathOperator{\bm}{bm}
\DeclareMathOperator{\me}{me}

\newcommand{\R}{\mathbb R}
\newcommand{\ra}{\rightarrow}

\begin{document}
\title[Martin boundary of random walks in convex cones]
{Martin boundary of random walks in convex cones}
\thanks{This project has received funding from the European Research Council (ERC) under the European Union's Horizon 2020 research and innovation programme under the Grant Agreement No 759702.}
\thanks{\version}

\author[J. Duraj]{Jetlir Duraj}
\address{Department of Economics, Harvard University, USA}
\email{duraj@g.harvard.edu}

\author[K. Raschel]{Kilian Raschel} 
\address{CNRS, Institut Denis Poisson, Universit\'e de Tours and Universit\'e d'Orl\'eans, France}
\email{raschel@math.cnrs.fr}
       
\author[P. Tarrago]{Pierre Tarrago} 
\address{Laboratoire de Probabilit\'es, Statistique et Mod\'elisation, Sorbonne Universit\'e, France}
\email{pierre.tarrago@cimat.mx}

\author[V. Wachtel]{Vitali Wachtel} 
\address{Institut f\"ur Mathematik, Universit\"at Augsburg, Germany}
\email{vitali.wachtel@mathematik.uni-augsburg.de}

\begin{abstract}
We determine the asymptotic behavior of the Green function for zero-drift random walks confined to multidimensional convex cones. As a consequence, we prove that there is a unique positive discrete harmonic function for these processes (up to a multiplicative constant); in other words, the Martin boundary reduces to a singleton.
\end{abstract}


\maketitle
\setcounter{tocdepth}{1}
\tableofcontents

\section{Introduction and main results}
\label{sec:intro}

The primary motivation of the present paper is to solve the following uniqueness problem for discrete harmonic functions: take a lattice $\Lambda$ (a linear transform of ${\bf Z}^d$), a convex cone $K$ in ${\bf R}^d$ and a discrete Laplacian operator
\begin{equation*}
     L(f)(x)=\sum_{x\sim y} p_{y-x} (f(y)-f(x)),
\end{equation*}
where the weights $\{p_{z}\}_{z\in \Lambda}$ sum to $1$, have zero drift (meaning that $\sum_{z\in\Lambda}zp_{z}=0$) and satisfy some minimal moment assumptions (we will be more specific later). We prove that up to multiplicative constants, there is a unique function $f:\Lambda \to {\bf R}$ which is positive, harmonic in $\Lambda\cap K$, i.e., $L(f)=0$, and equal to zero outside $K$. In terms of potential theory for random walks, we show that the Martin boundary of killed, zero-mean random walks in cones is reduced to one point. Our solution to this uniqueness problem is fully based on Martin boundary theory and requires the thorough asymptotic computation of the Green function for killed random walks in multidimensional cones. These asymptotics represent actually the main contribution of the paper.

\subsubsection*{Green functions and Martin boundary of random walks in cones}
Random walks conditioned to stay in multidimensional cones are a very popular topic in probability. Indeed, they appear naturally in various situations: nonintersecting paths \cite{St-90,EiKo-08,DeWa-10}, which can be seen as random walks in Weyl chambers, random walks in the quarter plane \cite{FaIaMa-17,Ra-11}, queueing theory \cite{CoBo-83}, branching processes and random walks in random environment \cite{AfGeKeVa-05}, finance \cite{CodL-13}, modelling of some populations in biology \cite{BiTr-12}, etc. As these random walk models are in bijection with many other discrete models (maps, permutations, trees, Young tableaux, partitions), they are also intensively studied in combinatorics \cite{BMMi-10,BoBMKaMe-16,DrHaRoSi-18}.

Let us now briefly review the literature regarding asymptotics of Green functions and Martin boundary for killed random walks in cones (see \cite{Sa-97} for a general introduction to Martin boundary theory). In the one-dimensional case, Doney \cite{Do-98} describes the harmonic functions and the Martin boundary of a random walk $\{S(n)\}$ on ${\bf Z}$ killed on the negative half-line (obviously there is essentially a unique cone in dimension $1$, namely ${\bf N}=\{0,1,2,\ldots \}$). Alili and Doney \cite{AlDo-01} extend this result to the corresponding space-time random walk $\{(S(n),n)\}$. 

In the higher dimensional case, let us start by quoting the famous Ney and Spitzer result \cite{NeSp-66} on the Green function asymptotics of drifted, unconstrained random walks in ${\bf Z}^d$. As a consequence, the Martin boundary is shown to be homeomorphic to the unit sphere ${\bf S}^{d-1}$. By large deviation techniques and Harnack inequalities, Ignatiouk-Robert \cite{IR-08,IR-09}, then Ignatiouk-Robert and Loree \cite{IRLo-10}, find the Martin boundary of random walks in half-spaces ${\bf N}\times {\bf Z}^{d-1}$ and orthants ${\bf N}^d$, with non-zero drift and killing at the boundary; they also derive the asymptotics of ratios of Green functions. For small step walks in the quarter plane, Lecouvey and Raschel \cite{LeRa-16} show that generating functions of harmonic functions are strongly related to certain conformal mappings.

The results on Green functions and Martin boundaries are rarer for driftless random walks, and typically require a strong underlying structure: the random walks are Cartesian products in \cite{PiWo-92}; they are associated with Lie algebras in \cite{Bi-91,Bi-92}; certain reflection groups are supposed to be finite in \cite{BiBoOC-05}. Varopoulos \cite{Va-99,Va-09} derives upper and lower bounds for the tail of the survival probability in cones under the assumption that the increments of the random walk are bounded. He also proves various statements on the growth or harmonic functions. Raschel \cite{Ra-11,Ra-14} obtains the asymptotics of the Green function and the Martin boundary in the case of small step quadrant random walks related to finite reflection groups. Bouaziz, Mustapha and Sifi \cite{BoMuSi-15} prove the existence and uniqueness of the positive harmonic function for random walks satisfying finite range, centering and ellipticity conditions, killed at the boundary of the orthant ${\bf N}^d$. Mustapha and Sifi \cite{MuSi-19} extend these results to Lipshitz domains, under similar hypotheses.
Ignatiouk-Robert \cite{IR-20} shows the uniqueness of the harmonic function in a convex cone, under the assumption that the first exit time has infinite expectation. Finally, in the paper \cite{RaTa-18}, the second and third authors derive a local limit theorem for zero-drift random walks confined to multidimensional convex cones, when the endpoint is close to the boundary. 

As we will see below, our theorems unify and extend all these results in the context of convex cones, under optimal moment assumptions.

\subsubsection*{Exit time, Green functions, harmonic functions and reverse random walk}
Consider a random walk $\{S(n)\}_{n\geq1}$ on ${\bf R}^d$, $d\geq1$, where
\begin{equation*}
     S(n) = X(1)+\cdots +X(n)
\end{equation*}
and $\{X(n)\}_{n\geq1}$ is a family of independent and identically distributed (i.i.d) copies of a random variable $X=(X_1,\ldots,X_d)$. The support of the increments is supposed to generate a lattice, which we denote by $\Lambda$.

Given a cone $K$, let $\tau_x$ be the first exit time from the cone $K$ of the random walk with starting point $x\in K$, i.e.,
\begin{equation}
\label{eq:exit_time_def}
     \tau_x=\inf\{n\geq 1 : x+S(n)\notin K\}.
\end{equation} 
By definition, the Green function of $S(n)$ killed at $\tau_x$ is 
\begin{equation}
\label{eq:Green_function_def}
     G_K(x,y) = \sum_{n=0}^\infty \pr(x+S(n)=y,\tau_x>n).
\end{equation}
A  function $h:{K}\to {\bf R}$ is said to be (discrete) harmonic with respect to ${K}$ and $\{S(n)\}$ if for every $x\in{K}$ and $n\geq 1$,
\begin{equation*}
     h(x)=\e[h(x+S(n)),\tau_{x}>n].
\end{equation*}
Remark that the above identity for $n=1$ implies all the other relations for $n\geq 2$. In the sequel, a harmonic function with respect to ${K}$ and $\{S(n)\}$ will be simply called a harmonic function.

Denisov and Wachtel proved \cite{DeWa-15,DeWa-19} the existence of a positive harmonic function $V:{K}\to {\bf R}_+$ defined by 
\begin{equation}\label{definition_V(x)}
     V(x)=\lim\limits_{n\rightarrow \infty}\e[u(x+S(n)),\tau_{x}>n].
\end{equation}
This harmonic function is of central importance in the present paper, since it will ultimately be identified with the Martin boundary of the random walk in ${K}$.

We denote by $\{S'(n)\}_{n\geq 1}$ the reverse random walk, which is the sum of the increments $\{X'(n)\}_{n\geq 1}$, i.i.d, independent from $\{X(n)\}_{n\geq 1}$ and such that $X'(n)$ is distributed as $-X$. In the sequel, every quantity involving $S'$ will be denoted similarly as the same quantity involving $S$, with a prime added at the right.

\subsubsection*{Notations and assumptions on cones and random walks}
Our hypotheses are of three types: some of them only concern the random walk (see \ref{H:drift_zero}, \ref{H:cov_id} and \ref{H:RW_aperiodic}), the assumption \ref{H:regularity_cone} is a convexity restriction on the cone, while the last ones, namely, \ref{H:strongly_irreducible}, \ref{H:moments} and \ref{local_assump} (moment assumptions) concern the behavior of the random walk in the cone.
\begin{enumerate}[label=($H\arabic{*}$),ref=($H\arabic{*}$)]
     \item\label{H:drift_zero}$\e[X_{i}]=0$ (zero drift assumption),
     \item\label{H:cov_id}$\text{cov}(X_{i},X_{j})=\delta_{i,j}$ (identity covariance matrix assumption),
     \item\label{H:RW_aperiodic}the random walk is strongly aperiodic, i.e., if $A=\lbrace x\in \Lambda: \mathbf{P}(X=x)>0\rbrace$, then $z+A$ generates $\Lambda$ for all $z\in\Lambda$.
\end{enumerate}     
Notice that \ref{H:cov_id} is not a restriction: we may always perform a linear transform so as to decorrelate the random walk (obviously this linear transform impacts on the cone in which the walk is defined). 

Denote by ${\bf S}^{d-1}$ the unit sphere of ${\bf R}^d$ and by $\Sigma$ an open, connected subset of ${\bf S}^{d-1}$. Let $K$ be the cone generated by the rays emanating from the origin and passing through $\Sigma$, i.e., $\Sigma=K\cap {\bf S}^{d-1}$; see Figure \ref{fig:cones} for two examples. In this paper, we shall suppose that
\begin{enumerate}[label=($H\arabic{*}$),ref=($H\arabic{*}$)]
\setcounter{enumi}{3}
     \item\label{H:regularity_cone}the cone $K$ is convex.
     \end{enumerate}

We further require a form of irreducibility of the random walk, which is an adaptation to unbounded random walks of the concept of reachability condition from infinity introduced in \cite{BoBMMe-18}. 
\begin{enumerate}[label=($H\arabic{*}$),ref=($H\arabic{*}$)]
\setcounter{enumi}{4}
     \item\label{H:strongly_irreducible} The random walk $S$ is asymptotically strongly irreducible, meaning that there exists a constant $R>0$ such that for any $z\in K\cap \Lambda$ with $\vert z\vert\geq R$, there exists a path with positive probability in $K\cap B(z,R)$ which starts in $z+K$ and ends at $z$.
\end{enumerate}
 There are several simple situations where the latter condition is satisfied, in particular when $\pr(X\in -K)>0$. If $K$ is $\mathcal{C}^{2}$, the condition \ref{H:strongly_irreducible} is superfluous.

When $K$ is convex, on each point $q$ of $\partial \Sigma$ there exists a non-trivial closed ball $B$ in ${\bf S}^{d-1}$ such that $B\cap \Sigma=q$. Hence, by standard analytic results \cite[Thm~6.13]{GiTh}, $\Sigma$ is regular for the Dirichlet problem. In particular (see for example the introduction of \cite{BaSm-97}), there exists a function $u$ harmonic on ${K}$, i.e., $\Delta u=0$, such that $u$ is positive in $K$ and $u_{\partial {K}}=0$, $\partial {K}$ denoting the boundary of $K$. This function is unique up to scalar multiplication, see \cite[Cor.~6.10 and Rem.~6.11]{GySal-11}, and is called the r\'eduite of ${K}$. It is homogeneous (or radial) in the sense that $u(tx)=t^{p}u(x)$ for all $t>0$ and $x\in K$. The homogeneity exponent $p$ is called the exponent of the cone ${K}$.

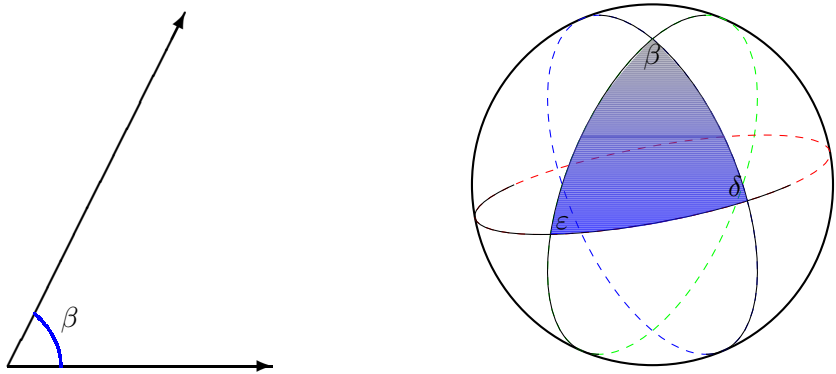
\begin{figure}[ht]
\begin{picture}(0,0)
    \thicklines
    \put(0,0){\textcolor{black}{\vector(1,0){100}}}
    \put(0,0){\textcolor{black}{\vector(1,2){67}}}
    \thicklines
    \put(20,15){$\beta$}
    \textcolor{blue}{\qbezier(20,0)(20,11)(10,20)}
    \end{picture} \qquad\qquad\qquad\qquad\qquad\qquad\qquad
\tdplotsetmaincoords{90}{90}
\begin{tikzpicture}[scale=3,tdplot_main_coords]

\tdplotsetthetaplanecoords{90}
\tdplotdrawarc[tdplot_rotated_coords,thick]{(0,0,0)}{0.8}{0}{360}{}{}
\tdplotsetrotatedcoords{60}{70}{0}
\tdplotdrawarc[dashed,tdplot_rotated_coords,name path=blue,color=blue]{(0,0,0)}{0.8}{0}{360}{}{}
\tdplotdrawarc[tdplot_rotated_coords]{(0,0,0)}{0.8}{0}{180}{}{}
\tdplotsetrotatedcoords{120}{110}{0}
\tdplotdrawarc[dashed,tdplot_rotated_coords,name path=green,color=green]{(0,0,0)}{0.8}{0}{360}{}{}
\tdplotdrawarc[tdplot_rotated_coords]{(0,0,0)}{0.8}{0}{180}{}{}
\tdplotsetrotatedcoords{220}{16}{0}
\tdplotdrawarc[dashed,tdplot_rotated_coords,name path=red,color=red]{(0,0,0)}{0.8}{0}{360}{}{}
\tdplotdrawarc[tdplot_rotated_coords]{(0,0,0)}{0.8}{-90}{90}{}{}


\path [name intersections={of={green and blue}, total=\n}]  
\foreach \i in {1,...,\n}{(intersection-\i) circle [radius=0.5pt] coordinate(gb\i){}};

\path [name intersections={of={green and red}, total=\n}]  
\foreach \i in {1,...,\n} {(intersection-\i) circle [radius=0.5pt]coordinate(gr\i){}};

\path[name intersections={of={red and blue}, total=\n}]  
\foreach \i in {1,...,\n}{(intersection-\i) circle [radius=0.5pt]coordinate(rb\i){}};

\shade[top color=gray,bottom color=blue,opacity=0.5]  
(rb3) to[bend left=7] (gr1) to[bend left=17] (gb2) to[bend left=15] cycle;

\draw (gb2) node[below]{$\beta$};
\draw (rb3) node[above left]{$\delta$};
\draw (gr1) node[above right]{$\varepsilon$};
\end{tikzpicture}
\caption{In dimension $2$, $\Sigma$ is an arc of circle and the cone $K$ is a wedge of opening $\beta$. In dimension $3$, any section $\Sigma\subset{\bf S}^2$ defines a cone. The picture on the right gives the example of a spherical triangle on the sphere ${\bf S}^2$, corresponding to the orthant $K={\bf N}^3$ (after possible decorrelation of the coordinates, see \ref{H:cov_id}).}
\label{fig:cones}
\end{figure}

Our next assumption \ref{H:moments} involves the quantity 
\begin{equation*}
     q=\sup_{\sigma\in\partial\Sigma}q_{\sigma}\geq 1,
\end{equation*}
that we now define. For each point $\sigma\in\partial\Sigma$, we define 
\begin{equation}
\label{eq:def_K_sigma}
     K_{\sigma}:=\lbrace u\in{\bf R}^d: \exists t>0, \sigma+tu\in K\rbrace.
\end{equation}
By convexity of $K$, the set $K_{\sigma}$ is a convex cone, which represents the cone tangent to $K$ at $\sigma$.  
Let $q_{\sigma}$ denote the exponent of $K_{\sigma}$. Note that we always have $1\leq K_{\sigma}\leq p$, since $K\subset K_{\sigma}$ and $K_{\sigma}$ is included in a half-space. When $K$ is $\mathcal C^{2}$ at $\sigma$, $K_{\sigma}$ is precisely a half-space, which yields $q_{\sigma}=1$.

We shall also assume a moment condition on the increments, which depends on the asymptotic shape of the cone $K$:
\begin{enumerate}[label=($M\arabic{*}$),ref=($M\arabic{*}$)]
\setcounter{enumi}{0}
     \item\label{H:moments}$\e[\vert X\vert^{r(p)}]<\infty$ for some $r(p)>p+q+d-2+(2-p)^{+}$ and $\e[\vert X\vert^{2+\delta}]<\infty$ for some $\delta>0$. If the boundary of $K$ is $\mathcal C^{2}$ (which implies $q=1$), the strict inequality for $r(p)$ may be replaced by a weak inequality.   
\end{enumerate} 
In the case where the cone is $\mathcal C^{2}$ or when considering asymptotic results inside the cone, the latter moment condition can be replaced by the following assumption of the local structure of the distribution of the increments:
\begin{enumerate}[label=($M\arabic{*}$),ref=($M\arabic{*}$)]
\setcounter{enumi}{1}
\item\label{local_assump}
$\mathbf{P}(X=x)\le \vert x\vert^{-p-d+1}f(\vert x\vert)$
for some function $f$ which is decreasing and such that $u^{(3-p)\vee 1}f(u)\to 0$ as
$u\to\infty$.
\end{enumerate} 

In this paper, we do not require the existence of a bigger cone $K'$ with $\partial K\setminus\lbrace 0\rbrace\subset \text{int}(K')$, such that the r\'eduite $u$ can be extended to a harmonic function on $K'$. This necessary condition in \cite{DeWa-15} is removed in \cite{DeWa-19} under the moment assumption \ref{H:moments}.

\subsubsection*{Main results}

Our first main result is the asymptotics of the Green function \eqref{eq:Green_function_def} in the regime where the endpoint tends to infinity while staying far from the boundary.
\begin{theorem}
\label{thm:deep}
Set $r_1(p)=p+d-2+(2-p)^+$ and assume that either $\mathbf{E}\vert X\vert^{r_1(p)}$ is finite or \ref{local_assump} holds.
\begin{enumerate}[label=($\alph{*}$),ref=($\alph{*}$)]
     \item\label{itA}If there exists $\alpha>0$ such that $\vert y\vert\to\infty$ with $\dist(y,\partial K)\ge \alpha\vert y\vert$, then
\begin{equation}
\label{T1}
G_K(x,y)
\sim cV(x)\frac{u(y)}{\vert y\vert^{2p+d-2}}.
\end{equation}
     \item\label{itB}If $\mathbf{E}\vert X\vert^{r}$ is finite for some $r>r_1(p)$, there exists $\rho>0$ such that \eqref{T1} holds uniformly for $\vert y\vert\to\infty$ with $\dist(y,\partial K)\ge \vert y\vert ^{1-\rho}$.
\end{enumerate}
\end{theorem}
We will construct an example showing that the moment assumptions of Theorem \ref{thm:deep} are optimal (see Section \ref{sec:optimal}). We now turn to the Green function asymptotics along the boundary. In the case when the cone is a half-space, we obtain the following:

\begin{theorem}
\label{thm:half-space}
Assume that $\mathbf{E}\vert X\vert^{d+1}<\infty$. Assume also that $x=(0,\ldots,0,x_d)$ with $x_d=o(\vert y\vert)$.
Then 
\begin{equation*}
G_K(x,y)\sim c\frac{V(x)V'(y)}{\vert y\vert^d}.
\end{equation*}
Here, $V'$ is the harmonic function for the killed reversed random walk $\{-S(n)\}$.
\end{theorem}

Theorem \ref{thm:half-space} appears to be not only an extension of Uchiyama's results \cite{Uchiyama14}, but will be one of the crucial tools to derive the boundary asymptotics of the Green function in the general convex case (Theorem \ref{thm:asymp_Green_boundary} below).

When $K$ is not a half-space but a general convex cone, we first introduce
\begin{equation}
\label{eq:K_rho}
     {K}_{\rho}:=\lbrace y\in{K}:\dist(y,\partial {K})\geq R\vert y\vert^{1-\rho}\rbrace
\end{equation}
as well as the stopping time
\begin{equation}
\label{eq:stopping_time_theta}
     \theta_{y}=\inf\{n\geq 1: y+S'(n)\in {K}_{\rho}\},
\end{equation}
for $y\in{K}$. We denote by $y_{\rho}$ the random element $y+S'(\theta_{y})$.

\begin{theorem}
\label{thm:asymp_Green_boundary}
Suppose that $\vert y\vert$ goes to infinity with $y/\vert y\vert$ converging to $\sigma\in \partial \Sigma$. Assume \ref{H:drift_zero}--\ref{H:strongly_irreducible} and  $\mathbf{E}\vert X\vert^{r(p)}<\infty$ for some $r(p)>p+q_{\sigma}+d-2+(2-p)^{+}$, then
\begin{equation*}
     G_K(x,y)\sim \frac{V(x) \e[u(y_{\rho}),\tau'_{y}>\theta_{y}]}{\vert y\vert^{p+q+d-2}}.
\end{equation*}
If $q_{\sigma}=1$, then the latter asymptotics can be improved as 
\begin{equation*}
     G_K(x,y)\sim \frac{V(x)c_{\sigma}(\dist(y,\partial K))}{\vert y\vert^{p+d-1}},
\end{equation*}
with $c_{\sigma}$ a positive function which is asymptotically linear, and the moment assumption can be replaced by $\mathbf{E}\vert X\vert^{p+d-1+(2-p)^{+}}<\infty$ or by \ref{local_assump}.
\end{theorem}
Let us comment on three different aspects of Theorem \ref{thm:asymp_Green_boundary}. First, we will construct an example showing that our hypotheses are optimal (see Section \ref{sec:optimal}). Moreover, in the above result, the convergence is uniform on all $\sigma\in \Sigma$. Finally, Theorem \ref{thm:asymp_Green_boundary} easily implies the identification of the Martin boundary of $S$ killed when exiting ${K}$, answering the uniqueness problem of the discrete harmonic functions.
\begin{theorem}\label{thm:Martin_boundary}
Assume \ref{H:drift_zero}--\ref{H:strongly_irreducible} and \ref{H:moments}. The Martin kernel of $S$ killed on the boundary of ${K}$ is reduced to one point, which corresponds to the function $V$ in \eqref{definition_V(x)}. In particular, there is up to a scaling constant a unique positive harmonic function killed at the boundary of ${K}$. If $K$ is $\mathcal C^{2}$, \ref{H:moments} can be changed into the local condition \ref{local_assump}.
\end{theorem}

\subsubsection*{Towards a Ney and Spitzer theorem in cones}

Ney and Spitzer consider in \cite{NeSp-66} random walks with non-zero drift in ${\bf Z}^d$ and prove that the Martin boundary is homeomorphic to the unit sphere ${\bf S}^{d-1}$. In \cite{IRLo-10,IR-09}, Ignatiouk-Robert and Loree prove that for random walks in ${\bf N}^d$ with a drift whose all entries are non-zero, the Martin boundary is homeomorphic to ${\bf S}^{d-1}\cap {\bf R}_+^d$. However, the question of a general non-zero drift (i.e., with zero entries allowed) is left opened in \cite{IRLo-10,IR-09}. Our results should allow to complete the picture; this will be the topic of future research.

\subsubsection*{Description of the methods used in our proofs}One of the standard approaches to the analysis of Green functions is based on local limit theorems for the process under consideration. For random walks confined to cones, one can apply local limit theorems from \cite{DeWa-15}. Since these results are applicable for $n>\varepsilon\vert y\vert^2$ only, one gets an asymptotically sharp lower estimate for $G_K$. To obtain an upper bound one needs good control over $\mathbf{P}(x+S(n)=y,\tau_x>n)$ for $n\ll \vert y\vert^2$. 
Caravenna and Doney \cite{Caravenna-Doney16} have used this approach to obtain necessary and sufficient conditions for validity of the local renewal theorem for one-dimensional non-restricted random walks. 
To control local large deviation probabilities in our model, we use recent results obtained in Raschel and Tarrago~\cite{RaTa-18} by using heat kernel estimates. These results, which are improvements of the local limit theorems of \cite{DeWa-15}, lead to Theorem~\ref{thm:deep} \ref{itB} and to the first claim in Theorem~\ref{thm:asymp_Green_boundary}. The analysis of local probabilities requires a slightly stronger moment assumptions than in Theorem~\ref{thm:deep} \ref{itA} and in the second half of Theorem~\ref{thm:asymp_Green_boundary} correspondingly. In order to derive these results, we use a different approach, where we control the whole
sum $\sum_{n\le\varepsilon \vert y\vert^2}\mathbf{P}(x+S(n)=y,\tau_x>n)$ instead of controlling every summand. This part is based on the functional limit theorem for walks in cones obtained in Duraj and Wachtel~\cite{DuWa-20}. As we have mentioned before, this approach requires less moments, but one needs to impose stronger regularity conditions on the boundary of the cone.

\subsubsection*{Structure and sketch of the results}
Our paper is organized as follows:
\begin{itemize}
     \item Section \ref{sec:interior}: proof of Theorem \ref{thm:deep} on the Green function asymptotics in the interior domain.
     \item Section \ref{sec:half-space}: proof of Theorem \ref{thm:half-space} on the Green function asymptotics along the boundary in the case of the half-space; this result has its own interest and will also be crucially used in the next section, in the general convex case.
     \item Section \ref{sec:boundary}: proof of Theorem \ref{thm:asymp_Green_boundary} on the Green function asymptotics along the boundary in the general case; proof of Theorem \ref{thm:Martin_boundary} on the structure of the Martin boundary (uniqueness problem).
     \item Section \ref{sec:optimal}: optimality of the moment assumptions in Theorems \ref{thm:deep} and \ref{thm:asymp_Green_boundary}.
     \item Section \ref{sec:survival}: proof of various lower bounds on the survival probability, which are used when showing Theorem \ref{thm:asymp_Green_boundary}.
\end{itemize}

\subsubsection*{Acknowledgments} 
KR would like to thank Rodolphe Garbit, Irina Ignatiouk-Robert and Sami Mustapha for various discussions concerning Martin boundary and the uniqueness problem for harmonic functions.

\section{Asymptotics of the Green function far from the boundary}
\label{sec:interior}

In this section, we prove Theorem \ref{thm:deep}.

\subsubsection*{Sketch of the proof} 
The proof runs as follows. Fix some $\varepsilon>0$ and split $G_K(x,y)$ into two parts:
\begin{align*}
G_K(x,y)&=\sum_{n<\varepsilon\vert y\vert^2}\mathbf{P}(x+S(n)=y,\tau_x>n)
+\sum_{n\ge\varepsilon \vert y\vert^2}\mathbf{P}(x+S(n)=y,\tau_x>n)\\
&=:S_1(x,y,\varepsilon)+S_2(x,y,\varepsilon).
\end{align*}
The main idea is that the first term will be negligible, meaning that
\begin{equation}
\label{lim_S_1}
     \lim_{\varepsilon\to0}\limsup_{\vert y\vert\to\infty}\frac{\vert y\vert^{2p+d-2}}{u(y)}S_1(x,y,\varepsilon)=0,
\end{equation}
while the second term $S_2(x,y,\varepsilon)$ will provide the main contribution in the Green function asymptotics. The asymptotic analysis \eqref{lim_S_1} of $S_1(x,y,\varepsilon)$ is very different under the hypotheses \ref{itA} and \ref{itB} of Theorem~\ref{thm:deep}; on the contrary, the study of $S_2(x,y,\varepsilon)$ will be done uniformly in the two cases.

\subsubsection*{Asymptotics of $S_2(x,y,\varepsilon)$}

Let $\rho>0$ small enough (to be chosen later), $A=\epsilon^{-1}$ and
\begin{equation*}
     K_{n,\epsilon}^{A}=\{z\in K: \vert z\vert \leq A\sqrt{n},\, \dist(z,\partial K)\geq n^{1-\epsilon}\}.
\end{equation*}
By \cite[Prop.~9]{RaTa-18}, 
\begin{equation*}
\frac{n^{p/2+d/2}}{u\left(\frac{y}{\sqrt{n}}\right)}\mathbf{P}(x+S(n)=y,\tau_x>n)=
\varkappa H_0 V(x)e^{-\vert y\vert^2/2n}+o(1)
\end{equation*} 
uniformly in $y\in K_{n,\epsilon}^{A}$, and by \cite[Thm~5]{DeWa-15}, 
\begin{equation*}
n^{p/2+d/2}\mathbf{P}(x+S(n)=y,\tau_x>n)=
\varkappa H_0 V(x)u\left(\frac{y}{\sqrt{n}}\right)e^{-\vert y\vert^2/2n}+o(1)
\end{equation*}
uniformly in $y\in K$. As $\vert y\vert\to\infty$ with $\dist(y,\partial K)\geq \vert y\vert^{1-\rho}$ and $\rho$ small enough, one has $y\in K_{n,\epsilon}^{A}$ for all $\epsilon\vert y\vert^{2}\leq n\leq \vert y\vert^{2+\epsilon'}$, for some $\epsilon'>0$. Hence,
\begin{align*}
S_2(x,y,\varepsilon)
&=\varkappa H_0 V(x)\sum_{ \varepsilon\vert y\vert^2\leq n\leq \vert y\vert^{2+\epsilon'}}\frac{1}{n^{p/2+d/2}}
u\Bigl(\frac{y}{\sqrt{n}}\Bigr)e^{-\vert y\vert^2/2n}\\
&\quad\quad+o\left(u(y)\sum_{\varepsilon\vert y\vert^2\leq n\leq \vert y\vert^{2+\epsilon'}}n^{-p-d/2}\right)+
o\left(\sum_{n\ge \vert y\vert^{2+\epsilon'}}\frac{1}{n^{p/2+d/2}}\right)\\
&=\varkappa H_0 V(x)u(y)\sum_{n\ge \varepsilon\vert y\vert^2}\frac{1}{n^{p+d/2}}e^{-\vert y\vert^2/2n}\\
&\quad\quad+
o\left((u(y)\vert y\vert^{-p}+\vert y\vert^{-\epsilon'})\vert y\vert^{-p-d+2}\right)\\
&=\varkappa H_0 V(x)u(y)\vert y\vert^{-2p-d+2}\int_\varepsilon^\infty z^{-p-d/2}e^{-1/(2z)}dz\\
&\quad\quad
+o\left((u(y)\vert y\vert^{-p}+\vert y\vert^{-\epsilon'})\vert y\vert^{-p-d+2}\right).
\end{align*}
Letting here $\varepsilon\to0$ and recalling that $u(y)\ge c\vert y\vert^{p-\rho}$ for $\dist(y,\partial K)\ge \vert y\vert^{1-\rho}$, see \cite[Lem.~19]{DeWa-15} and \cite{Va-99}, we obtain that for $\rho$ small enough,
\begin{equation}
\label{lim_S_2}
\lim_{\varepsilon\to0}\lim_{\substack{\vert y\vert\to\infty\\ \dist(y,\partial K)>\vert y\vert^{1-\rho}}}\frac{\vert y\vert^{2p+d-2}}{u(y)}S_2(x,y,\varepsilon)
=\varkappa H_0 V(x)\int_0^\infty z^{-p-d/2}e^{-1/(2z)}dz.
\end{equation}

\subsubsection*{Asymptotics of $S_1(x,y,\varepsilon)$ in case \ref{itA}}

Let us prove the first part of Theorem \ref{thm:deep}. It remains to show that \eqref{lim_S_1} holds.
Fix additionally some small $\delta>0$ and define
\begin{equation*}
     \Theta_y:=\inf\{n\ge1:x+S(n)\in B_{\delta,y}\},
\end{equation*}
where $B_{\delta,y}$ denotes the ball of radius $\delta\vert y\vert$ around the point $y$. Then we have
\begin{align}
\label{S1_estim_1}
\nonumber
&S_1(x,y,\varepsilon)=\sum_{n<\varepsilon\vert y\vert^2}\mathbf{P}(x+S(n)=y,\tau_x>n\ge\Theta_y)\\
\nonumber
&=\sum_{n<\varepsilon\vert y\vert^2}\sum_{k=1}^n\sum_{z\in B_{\delta,y}}
\mathbf{P}(x+S(k)=z,\tau_x>k=\Theta_y)\mathbf{P}(z+S(n-k)=y,\tau_z>n-k)\\
\nonumber
&\le \sum_{k<\varepsilon\vert y\vert^2}\sum_{z\in B_{\delta,y}}
\mathbf{P}(x+S(k)=z,\tau_x>k=\Theta_y)\sum_{j<\varepsilon\vert y\vert^2-k}\mathbf{P}(z+S(j)=y)\\
&\le\mathbf{E}\left[G^{(\varepsilon\vert y\vert^2)}(y-x-S(\Theta_y));\tau_x>\Theta_y,\Theta_y\le\varepsilon\vert y\vert^2\right],
\end{align}
where
\begin{equation*}
G^{(t)}(z):=\sum_{n<t}\mathbf{P}(S(n)=z).
\end{equation*}

We first focus on the case $d\ge3$. Then, according to \cite[Thm~2]{Uc-98}, for all $z\in{\bf Z}^d$,
\begin{equation}
\label{Green_full}
     G(z):=G^{(\infty)}(z)\le\frac{C}{1+\vert z\vert^{d-2}},
\end{equation}
provided that $\mathbf{E}\vert X_1\vert^{s_d}<\infty$, where $s_d=2+\varepsilon$ for $d=3,4$ and
$s_d=d-2$ for $d\ge5$. Since $r_1(p)=p+d-2+(2-p)^+>s_d$, \eqref{Green_full} yields
\begin{align}
\label{S1_estim_2}
\nonumber
&S_1(x,y,\varepsilon)\\
&\le C\mathbf{E}\left[\frac{1}{1+\vert y-x-S(\Theta_y)\vert^{d-2}};\tau_x>\Theta_y,\Theta_y\le\varepsilon\vert y\vert^2\right]\\
\nonumber
&\le C\mathbf{P}(\vert y-x-S(\Theta_y)\vert\le\delta^2\vert y\vert,\tau_x>\Theta_y,\Theta_y\le\varepsilon\vert y\vert^2)
+\frac{C(\delta)}{\vert y\vert^{d-2}}\mathbf{P}(\tau_x>\Theta_y,\Theta_y\le\varepsilon\vert y\vert^2).
\end{align}
Noting now that $\vert y-x-S(\Theta_y)\vert\le\delta^2\vert y\vert$ yields $\vert X(\Theta_y)\vert>\delta(1-\delta)\vert y\vert$
and using our moment assumption, we conclude that
\begin{align}
\label{S1_estim_3}
\nonumber
&\mathbf{P}(\vert y-x-S(\Theta_y)\vert \le\delta^2\vert y\vert ,\tau_x>\Theta_y,\Theta_y<\varepsilon\vert y\vert^2)\hspace{1cm}\\
\nonumber
&\hspace{1cm}\le \sum_{k<\varepsilon\vert y\vert^2}\mathbf{P}(\vert X(k)\vert >\delta(1-\delta)\vert y\vert ,\tau_x>k=\Theta_y)\\
\nonumber
&\hspace{1cm}\le \mathbf{P}(\vert X\vert >\delta(1-\delta)\vert y\vert)\sum_{k<\varepsilon\vert y\vert^2}\mathbf{P}(\tau_x>k-1)\\
&\hspace{1cm}=o\left(\vert y\vert^{-r_1(p)}\mathbf{E}[\tau_x;\tau_x<\vert y\vert^2]\right)=o\left(\vert y\vert^{-d-p+2}\right).
\end{align}
Recalling that $V$ is harmonic for $S(n)$ killed at leaving $K$, we obtain
\begin{align*}
&\mathbf{P}(\tau_x>\Theta_y,\Theta_y<\varepsilon\vert y\vert^2)\\
&\hspace{1cm}=\sum_{k<\varepsilon\vert y\vert^2}\sum_{z:\vert z-y\vert\le\delta\vert y\vert}\mathbf{P}(\tau_x>k,\Theta_y=k,x+S(k)=z)\\
&\hspace{1cm}=\sum_{k<\varepsilon\vert y\vert^2}\sum_{z:\vert z-y\vert\le\delta\vert y\vert}\frac{V(x)}{V(z)}\mathbf{P}^{(V)}(\Theta_y=k,x+S(k)=z)\\
&\hspace{1cm}\le\frac{V(x)}{\min_{z\in K:\vert z-y\vert\le\delta\vert y\vert}V(z)}\mathbf{P}^{(V)}(\Theta_y<\varepsilon\vert y\vert^2).
\end{align*}
It follows from the assumption $\dist(y,\partial K)\ge\alpha\vert y\vert$ and \cite[Lem.~13]{DeWa-15} that for sufficiently small $\delta>0$,
\begin{equation*}
\min_{z\in K:\vert z-y\vert\le\delta\vert y\vert}V(z)\ge C\vert y\vert^p.
\end{equation*}
As a result,
\begin{equation*}
\vert y\vert^p \mathbf{P}(\tau_x>\Theta_y,\Theta_y<\varepsilon\vert y\vert^2)
\le C(x)\mathbf{P}^{(V)}\left(\max_{n<\varepsilon\vert y\vert^2}\vert x+S(n)\vert>(1-\delta)\vert y\vert\right).
\end{equation*}
Applying now the functional limit theorem for $S(n)$ under $\mathbf{P}^{(V)}$, see Theorem 2 and Corollary 3
in \cite{DuWa-20}, we conclude that
\begin{align}
\label{S1_estim_4}
\lim_{\varepsilon\to0}\limsup_{\vert y\vert\to\infty} \vert y\vert^p \mathbf{P}(\tau_x>\Theta_y,\Theta_y<\varepsilon\vert y\vert^2)=0.
\end{align}
Note that the functional limit theorem from \cite{DuWa-20} only requires $p\vee (2+\epsilon)$-moments. Combining \eqref{S1_estim_2}--\eqref{S1_estim_4}, we infer that \eqref{lim_S_1} is valid under the 
assumption $\mathbf{E}\vert X_1\vert^{r_1(p)}<\infty$ in all dimensions $d\ge 3.$

Assume now that \ref{local_assump} holds. It is clear that this restriction implies $\mathbf{E}\vert X_1\vert^p<\infty$. Therefore, \cite[Thm~5]{DeWa-15} is still applicable and \eqref{lim_S_2}
remains valid for all random walks satisfying \ref{local_assump}. In order to show that \eqref{lim_S_1}
remains valid as well, we notice that
\begin{align*}
&S_1(x,y,\varepsilon)\\
&\hspace{0.5cm}\le C\mathbf{E}\left[\frac{1}{1+\vert y-x-S(\Theta_y)\vert^{d-2}};\vert y-x-S(\Theta_y)\vert\le\delta^2\vert y\vert,\tau_x>\Theta_y,\Theta_y\le\varepsilon\vert y\vert^2\right]\\
&\hspace{1.5cm}+\frac{C(\delta)}{\vert y\vert^{d-2}}\mathbf{P}(\tau_x>\Theta_y,\Theta_y\le\varepsilon\vert y\vert^2).
\end{align*}
In view of \eqref{S1_estim_4}, we have to estimate the first term on the right-hand side only. 
For any $z$ such that $\vert z-y\vert\le \delta^2\vert y\vert$ we have
\begin{align*}
&\mathbf{P}(x+S(\Theta_y)=z,\tau_x>\Theta_y,\Theta_y\le \varepsilon\vert y\vert^2)\\
&\hspace{1cm}\le\sum_{k=1}^{\varepsilon\vert y\vert^2}\sum_{z'\in K\setminus B_{\delta,y}}
\mathbf{P}(x+S(k-1)=z',\tau_x>k-1)\mathbf{P}(X(k)=z-z').
\end{align*}
Since $\vert z-z'\vert>\delta(1-\delta)\vert y\vert$, we infer from \ref{local_assump} that
\begin{align}
\label{local_1}
\nonumber
&\mathbf{P}(x+S(\Theta_y)=z,\tau_x>\Theta_y,\Theta_y\le \varepsilon\vert y\vert^2)\\
\nonumber
&\hspace{1cm}\le C(\delta)\vert y\vert^{-p-d+1}f(\delta(1-\delta)\vert y\vert)\sum_{k=1}^{\varepsilon\vert y\vert^2}\mathbf{P}(\tau_x>k-1)\\
&\hspace{1cm} \le C(\delta)\vert y\vert^{-p-d+1}f(\delta(1-\delta)\vert y\vert)\mathbf{E}[\tau_x;\tau_x<\vert y\vert^2].
\end{align}
Here and in the following we use that $\mathbf{E}[\tau_x;\tau_x<\vert y\vert^2]\sim C\vert y\vert^{-p+2}$ if $p\le 2$, as shown in \cite[Thm~1]{DeWa-15}. For every positive integer $m$, there are $O(m^{d-1})$ lattice points $z$ such that $\vert z-y\vert\in(m,m+1]$. Then, using \eqref{local_1}, we obtain
\begin{align*}
&\mathbf{E}\left[\frac{1}{1+\vert y-x-S(\Theta_y)\vert^{d-2}};\vert y-x-S(\Theta_y)\vert\le\delta^2\vert y\vert,\tau_x>\Theta_y,\Theta_y\le\varepsilon\vert y\vert^2\right]\\
&\hspace{1cm}\le C(\delta)\vert y\vert^{-p-d+1}f(\delta(1-\delta)\vert y\vert)\mathbf{E}[\tau_x;\tau_x<\vert y\vert^2]\sum_{m=1}^{\delta^2\vert y\vert}\frac{m^{d-1}}{1+m^{d-2}}\\
&\hspace{1cm} \le C(\delta)\vert y\vert^{-p-d+3}f(\delta(1-\delta)\vert y\vert)\mathbf{E}[\tau_x;\tau_x<\vert y\vert^2].
\end{align*}
Recalling that $u^{(3-p)\vee1}f(u)\to0$, we conclude that
\begin{multline*}
\mathbf{E}\left[\frac{1}{1+\vert y-x-S(\Theta_y)\vert^{d-2}};\vert y-x-S(\Theta_y)\vert\le\delta^2\vert y\vert,\tau_x>\Theta_y,\Theta_y\le\varepsilon\vert y\vert^2\right]\\
=o(\vert y\vert^{-p-d+2}).
\end{multline*}
This completes the proof of the theorem for $d\ge 3$.

We now focus on $d=2$; in this case, we cannot use the full Green function.
We will obtain bounds for $G^{(t)}(x)$ directly from the local limit theorem for unrestricted walks.
More precisely, we shall use Propositions 9 and 10 from Chapter 2 in Spitzer's book \cite{Sp-76}, which assert that as $n\rightarrow \infty$,
\begin{equation}
\label{eq_Spitzer}
     \mathbf{P}(S(n)=z)=\frac{1}{2\pi n}e^{-\vert z\vert^2/2n}+\frac{\rho(n,z)}{\vert z\vert^2\vee n},
\end{equation}
where as $n\rightarrow \infty$,
\begin{equation*}
\sup_{z\in{\bf Z}^2}\rho(n,z)\to0.
\end{equation*}
This asymptotic representation implies that for all $t\ge2$,
\begin{equation}
\label{Green_uniform}
\sup_{z\in{\bf Z}^2}G^{(t)}(z)\le C\log t.
\end{equation}
Furthermore, for $\vert z\vert\to\infty$ and $t\le a\vert z\vert^2$, one has
\begin{equation*}
     G^{(t)}(z)\le \sum_{n=1}^{a\vert z\vert^2}\frac{1}{2\pi n} e^{-\vert z\vert^2/2n}+o(1)
=\frac{1}{2\pi}\int_0^a\frac{1}{v}e^{-1/2v}dv+o(1).
\end{equation*}
As a result,
\begin{equation}
\label{Green_large}
\sup_{z\in{\bf Z}^2}G^{(a\vert z\vert^2)}(z)\le C(a)<\infty.
\end{equation}
Using \eqref{Green_uniform} and \eqref{Green_large}, we obtain
\begin{multline*}
S_1(x,y,\varepsilon)
\le C\log\vert y\vert\mathbf{P}(\vert y-x-S(\Theta_y)\vert\le\delta^2\vert y\vert,\tau_x>\Theta_y,\Theta_y\le\varepsilon\vert y\vert^2)\\
+C(\varepsilon)\mathbf{P}(\tau_x>\Theta_y,\Theta_y\le\varepsilon\vert y\vert^2).
\end{multline*}
According to \eqref{S1_estim_3},
\begin{align*}
\mathbf{P}(\vert y-x-S(\Theta_y)\vert\le\delta^2\vert y\vert,\tau_x>\Theta_y,\Theta_y\le\varepsilon\vert y\vert^2)
&=o(\vert y\vert^{-r_1(p)}\mathbf{E}[\tau_x;\tau_x<\vert y\vert^2])\\
&=o(\vert y\vert^{-p}/\log\vert y\vert).
\end{align*}
Combining this with \eqref{S1_estim_4}, we conclude that \eqref{lim_S_1} holds for $d=2$. The proof of Theorem~\ref{thm:deep} \ref{itA} is completed.

\subsubsection*{Preliminary estimates for the proof of Theorem~\ref{thm:deep} \ref{itB}}

In this part, we give some bounds on the local probability $\pr(x+S(n)=y,\tau_{x}>n)$, when $\vert x-y\vert$ is between the order of fluctuations $n^{1/2}$ and $n^{1/2+\kappa}$, for some $\kappa$ small enough. The main result will be given in Proposition \ref{bound_local_far_away}; it needs three lemmas, stated as Lemmas \ref{distribution_large_range}, \ref{boundAwayFixedX} and \ref{roughBoundLocal}.

We will use the coupling of Zaitsev and G\"otze (see \cite[Thm~4]{GoZa-09} and \cite[Lem.~17]{DeWa-15}) for random walks having increments satisfying to \ref{H:moments}. Suppose that $X$ has moments of order $r(p)$, with $r(p)>p+d-2+(2-p)^{+}$ and $r(p)>2+\delta$. By \cite[Thm~4]{Ei-89}, there exists a constant $K$ such that for $\gamma\leq 1/2-1/r(p)$,
\begin{equation}\label{couplingEinmahl}
\pr\left(\sup_{0\leq s\leq n}\vert S(\lfloor k\rfloor )-B(k)\vert \geq n^{1/2-\gamma}\right)\leq Kn^{-r},
\end{equation}
with 
\begin{equation}
\label{eq:def_local_r}
     r=r(p)(1/2-\gamma)-1.
\end{equation}

In the proof of the following lemma, we use several estimates from \cite{RaTa-18} on the transition probabilities of a Brownian motion in a cone. Those estimates come from general Gaussian estimates for the heat kernel in a Lipschitz domain, see \cite[Sec.~6]{GySal-11} for general statements. The first inequality from \cite[Thm~5.11]{GySal-11} gives an upper bound for the transition probabilities in $K$ for the Brownian motion started at $y\in K$ killed outside $K$:
\begin{equation}\label{estimateK}
\pr(y+ B(1)\in dz,\tau_{y}^{\bm}>1)\leq C\pr(\tau_{y}^{\bm}>1)\exp(-\vert z-y\vert^{2}/c)dz,
\end{equation}
for some positive constants $c$ and $C$. The survival time of the Brownian motion in $K$ is well estimated by the réduite, as shows the following inequality from \cite[Thm~5.4]{GySal-11}:
\begin{equation}
\label{estimatek}
     \pr(\tau_{y}^{\bm}>1)\leq Cu(y).
\end{equation}

Define 
\begin{equation*}
     K_{n,\epsilon}=K_{n,\epsilon}^{\infty}=\{z\in K: \dist(z,\partial K)\geq n^{1-\epsilon}\}.
\end{equation*}

\begin{lemma}\label{distribution_large_range}
There exist $\kappa,\epsilon,c,C>0$ such that for all $n$ large enough and $A\sqrt{n}\leq t\leq n^{1/2+\kappa}$, 
\begin{equation*}
     \pr(\vert S(n)\vert>t,\tau_{y}>n)\leq C\left(u(y/\sqrt{n})\exp(-t^{2}/(cn))+n^{-r}\right)
\end{equation*}
and for $y\in{K}_{n,\epsilon}$ such that $\vert y\vert\leq n^{1/2+\kappa}$,
\begin{equation*}\pr(\tau_{y}>n)\leq Cu(y/\sqrt{n}).\end{equation*}
\end{lemma}
\begin{proof}
Choose $x_{0}\in{\bf R}^{d}$ and $R_{0}>0$ such that
\begin{equation*}
     \vert x_{0}\vert =1,\quad x_{0}+{K}\subset {K}\quad \text{and}\quad \dist(Rx_{0}+{K}, \partial{K})>1.
\end{equation*}
Let $y\in {K}_{n,\epsilon}^{A}$ and set $y^{+}:=y+ Rx_{0}n^{1/2-\gamma}$. Using the same construction as in the proof of \cite[Lem.~20]{DeWa-15}, we get
\begin{multline*}
     \pr(\vert S(n)\vert>t,\tau_{y}>n)\\\leq \int_{\vert z-y/\sqrt{n}\vert>t/\sqrt{n}-2Rn^{-\gamma}}\pr(y^{+}/\sqrt{n}+B(1)\in dz,\tau_{y^{+}/\sqrt{n}}^{\bm}>1)+O(n^{-r}).
\end{multline*}
Using \eqref{estimateK} yields 
\begin{multline}
\int_{\vert z-y/\sqrt{n}\vert>t/\sqrt{n}-2Rn^{-\gamma}}\pr(y^{+}/\sqrt{n}+B(1)\in dz,\tau^{\bm}_{y^{+}/\sqrt{n}}>1)\\
\leq C\pr(\tau^{\bm}_{y^{+}/\sqrt{n}}>1)\int_{\vert z-y/\sqrt{n}\vert>t/\sqrt{n}-2Rn^{-\gamma}}C\exp(-\vert z-y^{+}/\sqrt{n}\vert^{2}/c)dz.\label{first_bound_probability_away}
\end{multline}
By the local H\"older continuity of  the survival probability $\pr(\tau_{x}^{\bm}>1)$ in $x$ (see \cite[Prop.~18]{RaTa-18}), there exist $\alpha,\chi,C_{\alpha}>0$ such that 
\begin{equation*}
     \pr(\tau^{\bm}_{y^{+}/\sqrt{n}}>1)\leq \pr(\tau^{\bm}_{y/\sqrt{n}}>1)+C_{\alpha}\left(\vert y\vert/\sqrt{n}\right)^{\chi}n^{-\alpha\gamma}.
\end{equation*}
Hence, using \eqref{estimatek} yields
\begin{equation*}
     \pr(\tau^{\bm}_{y^{+}/\sqrt{n}}>1)\leq Cu(y/\sqrt{n})+C_{\alpha}n^{\chi\kappa-\alpha\gamma}.
\end{equation*}
By \cite[Lem.~19]{DeWa-15}, one has
\begin{equation}\label{lowerBoundu}
u(x)\geq cd(x,\partial K)^{p},
\end{equation}
so that $u(y/\sqrt{n})\geq \dist(y/\sqrt{n},{K})^{p}\geq n^{-p\epsilon}$; choosing $\kappa$ such that $\alpha\gamma-\chi\kappa>0$ and then $\epsilon$ such that $\epsilon\leq (\alpha\gamma-\chi\kappa)/p$ yields that for some $C>0$ and $y\in{K}_{n,\epsilon}$ with $\vert y\vert\leq n^{1/2+\kappa}$,
\begin{equation*}
     \pr(\tau^{\bm}_{y^{+}/\sqrt{n}}>1)\leq Cu(y/\sqrt{n}).
\end{equation*}
Hence, integrating in \eqref{first_bound_probability_away} over the angular coordinates gives
\begin{multline*}
\int_{\vert z-y/\sqrt{n}\vert>t/\sqrt{n}-2Rn^{-\gamma}}\pr(y^{+}/\sqrt{n}+B(1)\in dz,\tau^{\bm}_{y^{+}/\sqrt{n}}>1)\\
\leq Cu(y/\sqrt{n})\int_{z>t/\sqrt{n}-4Rn^{-\gamma}}\exp(-\vert z\vert^{2}/c)dz
\end{multline*}
for some $C>0$. The latter inequality for $t=0$ gives the second inequality of Lemma~\ref{distribution_large_range}. For the first one, notice that there exists $C>0$ such that $\int_{x}^{\infty}\exp(-z^{2})dz\leq C\exp(-x^{2})$. Choosing $\kappa<\gamma$ yields $\exp((t/\sqrt{n}-4Rn^{-\gamma})^{2}/c)\sim \exp((t/\sqrt{n})^{2}/c)$ for $t\leq n^{1/2+\kappa}$, and finally we obtain that for some constant $C>0$,
\begin{equation*}
     \int_{\vert z-y/\sqrt{n}\vert>t/\sqrt{n}-2n^{-\gamma}}\pr(y^{+}/\sqrt{n}+B(1)\in dz,\tau^{\bm}_{y^{+}/\sqrt{n}}>1)\leq C u(y/\sqrt{n})\exp((t/\sqrt{n})^{2}/c).\qedhere
\end{equation*}
\end{proof}

We can extend the latter result by relaxing the condition $y\in{K}_{n,\epsilon}$.
\begin{lemma}\label{boundAwayFixedX}
Let $x\in{K}$. There exists $C>0$ such that for all $t\leq n^{1/2+\kappa}$ ($\kappa$ being as in Lemma \ref{distribution_large_range}),
\begin{equation*}
     \pr(\vert S(n)\vert \geq t, \tau_{x}\geq n)\leq C\left(V(x)n^{-p/2}\exp(-t^{2}/(cn))+ n^{-r}\right).
\end{equation*}
\end{lemma}

\begin{proof}
Introduce the stopping time
\begin{equation}
\label{eq:t_x}
     t_{x,\epsilon}(n)=\inf\{n\geq 1:x+S(n)\in {K}_{n,\epsilon}\}
\end{equation}
and $x_{\epsilon}(n)=x+S(t_{x,\epsilon}(n))$. Then, applying \cite[Sec.~4]{DeWa-15} to Lemma \ref{distribution_large_range}, we get
\begin{align*}
\pr(\vert S(n)\vert \geq t,& \tau_{x}\geq n)\leq C\Bigl(n^{-p/2}\exp(-t^{2}/(cn))\times\\
&\hspace{3cm}\e\left(u(x_{\epsilon}(n)),\tau_{x}>t_{x,\epsilon}(n),t_{x,\epsilon}(n)\leq n^{1-\epsilon}\right)+ n^{-r}\Bigr)\\
&+n^{-p/2}O\left(\e\left(\vert x_{\epsilon}(n)\vert^{p},\vert x_{\epsilon}(n)\vert> \theta_{n}\sqrt{n},\tau_{x}>t_{x,\epsilon}(n),t_{x,\epsilon}(n)\leq n^{1-\epsilon}\right)\right)\\
&+O(\exp(-Cn^{\epsilon'}),
\end{align*}
where $\theta_{n}=n^{-\epsilon/8}$ and $\epsilon'$ is small enough. Using Lemma \ref{Lemma_24_revisited} with $\alpha=p$ and $q=r(p)$ gives
\begin{equation*}n^{-p/2} \e\left(\vert x_{\epsilon}(n)\vert^{p},\vert x_{\epsilon}(n)\vert> \theta_{n}\sqrt{n},\tau_{x}>t_{x,\epsilon}(n),t_{x,\epsilon}(n)\leq n^{1-\epsilon}\right)=o(n^{-(p+d-2+(2-p)^{+})/2}),\end{equation*}
since $p+d-2+(2-p)^{+}<r(p)$. Since $(p+d-2+(2-p)^{+})/2\geq r$, see \eqref{eq:def_local_r}, Lemma~\ref{distribution_large_range} yields
\begin{equation*} n^{-p/2}\e\left(\vert x_{\epsilon}(n)\vert^{p},\vert x_{\epsilon}(n)\vert> \theta_{n}\sqrt{n},\tau_{x}>t_{x,\epsilon}(n),t_{x,\epsilon}(n)\leq n^{1-\epsilon}\right)=o(n^{-r})\end{equation*}
for $t\leq n^{1/2+\kappa}$. Since, by \cite[Lem.~21]{DeWa-15},
\begin{equation*}\lim_{n\rightarrow \infty}\e\left(u(x_{\epsilon}(n)),\tau_{x}>t_{x,\epsilon}(n),t_{x,\epsilon}(n)\leq n^{1-\epsilon}\right)=V(x),\end{equation*}
the result is deduced.
\end{proof}
For the next lemma, we need some bounds from \cite[Lem.~27 and 29]{DeWa-15}. There exist positive constants $a$ and $C$ such that for all $u\geq 0$,
\begin{equation}\label{bound_local_probability}
\limsup_{n\rightarrow \infty} n^{d/2}\sup_{\vert z-x\vert \geq u\sqrt{n}}\pr(x+S(n)=z)\leq C\exp(-au^{2}).
\end{equation}
In particular, there exists $C(x)>0$ such that 
\begin{equation}\label{bound_local_probability_in_cone}
\sup_{y\in{K}}\pr(x+S(n)=y,\tau_{x}\geq n)\leq C(x)n^{-p/2-d/2}.
\end{equation}
\begin{lemma}\label{roughBoundLocal}
There exist $C$ and $n_{0}$ such that for $n\geq n_{0}$, all $y\in {K}_{n,\epsilon}$ with $\vert y\vert\leq n^{1/2+\kappa}$ and all $z\in {K}$,
\begin{equation*}\pr(y+S(n)=z,\tau_{y}>n)\leq Cn^{-d/2}u(y/\sqrt{n}).\end{equation*}
\end{lemma}
\begin{proof}
Let $m:=\lfloor n/2\rfloor$. Then
\begin{align*}
\pr(y+S(n)=z,\tau_{y}>n)&=\sum_{z'\in{K}}\pr(y+S(m)=z',\tau_{y}>m)\pr(z'+S(n-m)=z,\tau_{z'}>n-m)\\
&\leq C\pr(\tau_{y}>m)m^{-d/2},
\end{align*}
where we have used \eqref{bound_local_probability} with $u=0$ to bound $\pr(z'+S(m)=z,\tau_{z'}>n-m)$. Thus, by Lemma \ref{distribution_large_range}, there exists $n_{0}$ such that for $n\geq n_{0}$ and $y\in{K}_{n,\epsilon}$ with $\vert y\vert \leq n^{1/2+\kappa}$,
\begin{equation*}
     \pr(y+S(n)=z,\tau_{y}>n)\leq Cn^{-d/2}u(y/\sqrt{n}).\qedhere
\end{equation*}
\end{proof}
Putting the previous results together yields the following estimate on the local probability at middle range.
\begin{proposition}
\label{bound_local_far_away}
Let $x\in{K}$. There exists $C$ such that 
\begin{equation*}
     \pr(x+S(n)=y,\tau_{x}>n)\leq CV(x)n^{-p/2-d/2}\left(u(y)n^{-p/2}\exp(-\vert x-y\vert^{2}/(c n))+ n^{-r}\right)
\end{equation*}
for all $y\in{K}_{n,\epsilon}$ such that $\vert y-x\vert\leq n^{1/2+\kappa}$.
\end{proposition}
\begin{proof}
Let $m:=\lfloor n/2\rfloor$. Then we have
\begin{align*}
\pr&(x+S(n)=y,\tau_{x}>n)\\
&=\sum_{z\in{K}:\,\vert z-x\vert \geq \vert y-x\vert/2}\pr(x+S(m)=z,\tau_{x}>m)\pr(y+S'(n-m)=y,\tau'_{y}>n-m)\\
&+\sum_{z\in{K}:\,\vert z-x\vert < \vert y-x\vert/2}\pr(x+S(m)=z,\tau_{x}>m)\pr(y+S'(n-m)=y,\tau'_{y}>n-m)\\
&=M_{1}+M_{2}.
\end{align*}
By Lemma \ref{boundAwayFixedX} and Lemma \ref{roughBoundLocal}, the first sum is bounded from above by
\begin{align*}
M_{1}&\leq Cu(y/\sqrt{n})n^{-d/2}\pr(\vert S(n)\vert>\vert x-y\vert/2,\tau_{x}>n)\\
&\leq  CV(x)u(y/\sqrt{n})n^{-d/2}\left(n^{-p/2}\exp(-\vert y-x\vert^{2}/(c n))+ n^{-r}\right),
\end{align*}
where we have used in the last inequality the hypothesis $\vert y-x\vert/2\leq n^{1/2+\kappa}$ in order to apply Lemma \ref{boundAwayFixedX}. Similarly, by \eqref{bound_local_probability_in_cone} and Lemma \ref{distribution_large_range}, the second sum is bounded by 
\begin{align*}
M_{2}&\leq CV(x)n^{-d/2-p/2}\pr(\vert S'(m)\vert>\vert x-y\vert/2,\tau'_{y}>n)\\
&\leq  CV(x)n^{-d/2-p/2}\left(u(y/\sqrt{n})\exp(-\vert y-x\vert^{2}/(c n))+ n^{-r}\right)\\
&\leq CV(x)u(y)n^{-d/2-p/2}\left(n^{-p/2}\exp(-\vert y-x\vert^{2}/(c n))+ n^{-r}\right),
\end{align*}
where we have used in the last inequality the hypothesis that $\vert y-x\vert/2\leq n^{1/2+\kappa}$ in order to apply Lemma \ref{distribution_large_range}, as well as the fact that $u(y)\geq 1$ for $y\in{K}_{n,\epsilon}$ and $n$ large enough. The result is then deduced by summing the bounds of $M_{1}$ and $M_{2}$.
\end{proof}


\subsubsection*{Asymptotics of $S_1(x,y,\varepsilon)$ in case \ref{itB}} 
We now prove Theorem \ref{thm:deep} under the hypothesis \ref{itB}. Without loss of generality, we assume that $d\geq 2$. We have to show that \eqref{lim_S_1} holds for $y$ satisfying $\dist(y,\partial K)\ge \vert y\vert^{1-\rho}$. Our strategy is to decompose $S_{1}(x,y,\varepsilon)$ as a sum of three terms:
\begin{equation}
\label{eq:3_term_dec_S1}
     S_{1}(x,y,\varepsilon)=\Sigma_{1}+\Sigma_{2}+\Sigma_{3}=\left(\sum_{n=0}^{N_{1}}+\sum_{n=N_{1}+1}^{N_{2}}+\sum_{n=N_{2}}^{\epsilon\vert x-y\vert^{2}}\right)\pr(x+S(n)=y,\tau_{x}>n),
\end{equation}
with $N_{1}$ of the form $\vert y-x\vert^{2-\nu}$ and $N_2$ to be defined later. We begin by giving an estimate of the truncated Green function $\Sigma_{1}$.

\begin{proposition}
\label{Green_very_large}
Let $\nu>0$ and suppose that $\e\vert X\vert^{r+(2-p)^{+}}<\infty$. Then, for all $a<r$,
\begin{equation*}
     \sum_{n=0}^{\vert y-x\vert^{2-\nu}}\pr(x+S(n)=y,\tau_{x}>n)=o(\vert x-y\vert^{a}).
\end{equation*}
\end{proposition}

\begin{proof}
Following the proof of \cite[Lem.~24]{DeWa-15}, we introduce the stopping time 
\begin{equation*}
     \mu=\inf\{i\geq 1: \vert X(i)\vert \geq \vert y-x\vert^{1-\nu/\alpha}\},
\end{equation*}
where $\alpha$ is large enough and will be chosen later. Let $n\leq \vert y-x\vert^{2-\nu}$. Then
\begin{multline*}
     \pr(x+S(n)=y,\tau_{x}>n)\\=\pr(x+S(n)=y,\tau_{x}>n,\mu>n)+\pr(x+S(n)=y,\tau_{x}>n,\mu\leq n).
\end{multline*}
On the one hand, using Fuk-Nagaev inequalities \cite{FuNa-71} as in \cite[Cor.~23]{DeWa-15} yields 
\begin{align*}
\pr(x+S(n)=y,\tau_{x}>n,\mu>n)&\leq \pr(\vert S(n)\vert \geq \vert x-y\vert/2,\,\sup_{k\leq n}\vert X(k)\vert \leq \vert y-x\vert^{1-\nu/\alpha})\\
&\leq \left(\frac{n\sqrt{d}e}{\vert x-y\vert^{2-\nu/\alpha}/2}\right)^{\vert x-y\vert^{\nu/\alpha}/(2\sqrt{d})}\\
&\leq \left(\frac{\vert x-y\vert^{2-\nu}\sqrt{d}e}{\vert x-y\vert^{2-\nu/\alpha}/2}\right)^{2\vert x-y\vert^{\nu/\alpha}/(2\sqrt{d})}\\&\leq\exp(-C\vert x-y\vert^{\nu/\alpha})
\end{align*}
for $y$ large enough. On the other hand, recall that since $X$ admits moments of order $r(p):=r+(2-p)^{+}$,
\begin{align*}
\pr(x+S(n)=y&,\tau_{x}>n,\mu\leq n)\\&\leq\sum_{k=1}^{n}\pr(\tau_{x}>k-1,\vert X(k)\vert \geq \vert y-x\vert^{1-\nu/\alpha},
y+S'(n-k)=x+S(k))\\
&\leq CV(x)\frac{\e[\vert X\vert^{r(p)}]}{\vert y-x\vert^{(1-\nu/\alpha)r(p)}}\sum_{k=1}^{n}k^{-p/2}(n+1-k)^{-d/2},
\end{align*}
where we have used the Markov property of the random walk, applied \eqref{bound_local_probability} with $u=0$ to $S'(n-k)$ and then \eqref{bound_local_probability_in_cone} in the last inequality. Hence, we get
\begin{align*}
\sum_{n=0}^{\vert y-x\vert^{2-\nu}}\pr(x+S(n)&=y,\tau_{x}>n)\leq\vert y-x\vert^{2-\nu}\exp(-C\vert x-y\vert^{\nu/\alpha})\\
&\ \ +CV(x)\frac{\e[\vert X\vert^{r(p)}]}{\vert y-x\vert^{(1-\nu/\alpha)r(p)}}\sum_{n=1}^{\vert y-x\vert^{2-\nu}}\sum_{k=1}^{n}k^{-p/2}(n+1-k)^{-d/2}\\
&\phantom{=y,\tau_{x}>n)\,\,}\leq C'\vert y-x\vert^{-(1-\nu/\alpha)r(p)}\sum_{k=1}^{\vert y-x\vert^{2-\nu}}k^{-p/2}\sum_{k=1}^{\vert y-x\vert^{2-\nu}}k^{-d/2}.
\end{align*}
Since $d\geq 2$, we have the elementary estimate, for some constant $C>0$,
\begin{equation*}
     \sum_{k=1}^{\vert y-x\vert^{2-\nu}}k^{-p/2}\sum_{k=1}^{\vert y-x\vert^{2-\nu}}k^{-d/2}\sim C\log\vert y-x\vert^{\mathbf{1}_{d=2}+\mathbf{1}_{p=2}}\left(\vert y-x\vert^{2-\nu}\right)^{(1-p/2)\wedge 0}.
\end{equation*}
Hence,
\begin{multline*}
     \sum_{n=0}^{\vert y-x\vert^{2-\nu}}\pr(x+S(n)=y,\tau_{x}>n)\\\leq C\vert y-x\vert^{-(1-\nu/\alpha)r(p)+(2-\nu)((1-p/2)\wedge 0)}\log\vert y-x\vert^{\mathbf{1}_{d=2}+\mathbf{1}_{p=2}}.
\end{multline*}
Finally, since $r(p)=r+(2-p)^{+}$, for $a<r$ and $\alpha$ large enough, we conclude the proof of Proposition \ref{Green_very_large}.
\end{proof}

We now conclude the proof of Theorem \ref{thm:deep} \ref{itB}.
\begin{lemma}
Suppose that $\e[\vert x\vert^{r(p)}]<\infty$ with $r(p)>p+d-2+(2-p)^{+}$. Then,
\begin{equation*}
     \lim_{\varepsilon\to0}\limsup_{\substack{\vert y\vert\to\infty\\ \dist(y,\partial K)>\vert y\vert^{1-\rho}}}\frac{\vert y\vert^{2p+d-2}}{u(y)}S_1(x,y,\varepsilon)=0.
\end{equation*}
\end{lemma}
\begin{proof}
Our starting point is the three-term decomposition \eqref{eq:3_term_dec_S1}. Since $\dist(y,\partial K)>\vert y\vert^{1-\rho}$, we have $u(y)\geq \vert y\vert^{p-p\rho}$ by \eqref{lowerBoundu}. Hence, it suffices to prove that
\begin{equation}
\label{eq:to_prove_1}
     \limsup_{\substack{\vert y\vert\to\infty\\ \dist(y,\partial K)>\vert y\vert^{1-\rho}}}\vert y\vert^{p+d-2+p\rho}(\Sigma_{1}+\Sigma_{2})=0
\end{equation}
and
\begin{equation}
\label{eq:to_prove_2}
     \lim_{\varepsilon\to0}\limsup_{\substack{\vert y\vert\to\infty\\ \dist(y,\partial K)>\vert y\vert^{1-\rho}}}\frac{\vert y\vert^{2p+d-2}}{u(y)}\Sigma_{3}=0.
\end{equation}

In order to prove \eqref{eq:to_prove_1}, we start by the following estimate, obtained in Proposition~\ref{Green_very_large}, for $\rho$ small enough:
\begin{equation*}
     \Sigma_{1}=o(\vert x-y\vert^{-p/2-d/2-p\rho}).
\end{equation*}
We now study $\Sigma_{2}$.
Let $\nu>0$ be such that $(2-\nu)(1/2+\kappa)> 1$, with $\kappa$ as in Lemma~\ref{distribution_large_range}. Suppose that $\delta<\nu$. With $c$ as in Lemma~\ref{distribution_large_range}, introduce
\begin{equation*}
     N_{2}=\inf\{n\geq 1: \exp\bigl(-\tfrac{\vert y-x\vert^{2}}{cn}\bigr)\geq n^{-r+p\rho}\}.
\end{equation*}
Recall that $r=r(p)(1/2-\gamma)-1$, see \eqref{eq:def_local_r}, and that $r>p/2$ for $d\geq 2$ and $\gamma$ small enough, so that $N_{2}$ exists as soon as $\rho$ is small enough. Furthermore, for $d\geq 2$ and $y$ large enough, $N_{2}\geq N_{1}$, since for $K$ large enough,
\begin{equation*}
     \exp\left(-\frac{\vert y-x\vert^{2}}{c\frac{\vert y-x\vert^{2}}{K\log \vert y-x\vert}}\right)=\vert y-x\vert^{-K/c}\leq\left(\frac{\vert y-x\vert^{2}}{K\log \vert x-y\vert}\right)^{-r+p/2}.
\end{equation*}
 

By our choice of $\nu$ and $\delta<\nu$, $\vert y-x\vert \leq n^{1/2+\kappa}$ for $n\geq \vert y-x\vert^{2-\delta}$ and $y$ large enough. Applying Proposition \ref{bound_local_far_away} to $\Sigma_{2}$ then yields
\begin{align*}
\Sigma_{2} & \leq CV(x)\epsilon\vert x-y\vert^{2}\left(\vert y-x\vert^{2-\nu}\right)^{-r-p/2-d/2}\\
& \leq C\varepsilon V(x)u(y)\vert y-x\vert^{-(2r+p+d-2)+f(\nu)}
\leq C\frac{V(x)}{A^{2}}\vert y-x\vert^{-(2r+p+d-2)+f(\nu)},
\end{align*}
where $f:{\bf R}\rightarrow{\bf R}$ is linear. Since $r>p\rho$ , choosing $\nu$ small enough yields 
\begin{equation*}\Sigma_{2}\leq C\varepsilon V(x)\vert y-x\vert^{-(r+p+d-2+u)},\end{equation*}
with $u=2r-p\rho>0$, for $y$ large enough. Hence \eqref{eq:to_prove_1} is proved.

We turn to the term $\Sigma_{3}$.
First, by the choice of $N_{2}$ and Proposition \ref{bound_local_far_away},
\begin{equation*}
     \Sigma_{3}\leq CV(x)u(y)\sum_{n=N_{2}+1}^{\varepsilon\lfloor  \vert y-x\vert^{2}\rfloor}n^{-p-d/2}\exp\left(-\frac{\vert y-x\vert^{2}}{cn}\right).
\end{equation*}
Set $g_{k,B}(t)=t^{-k}\exp(-B/t)$, with $B,k>0$. Then 
\begin{equation*}
     g_{k,B}'(t)=(Bt^{-k-2}-kt^{-k-1})\exp(-B/t),
\end{equation*}
and thus $g_{k,B}$ is increasing on $[0,B/k]$. Applying the latter property to $k=p+d/2$ and $B=\vert y-x\vert^{2}/c$ yields that if $\varepsilon^{-1}>c(p+d/2)$ (which we assume from now on), then
\begin{equation*}
     \Sigma_{3}\leq C\varepsilon\vert y\vert^{-2p-d+2}u(y)\exp(-\epsilon^{-1}/c).
\end{equation*}
This implies \eqref{eq:to_prove_2}, thereby completing the proof.
\end{proof}

\section{Boundary asymptotics of the Green function: the half-space case}
\label{sec:half-space}

In this section we shall consider a particular cone
\begin{equation*}
K=\bigl\{x\in{\bf R}^d:x_d>0\bigr\}.
\end{equation*}
Since the rotations of the space do not affect our moment assumptions, the results of this section
remain valid for any half-space in ${\bf R}^d$.
For this very particular cone, we have
\begin{itemize}
 \item $u(x)=x_d$;
 \item $\tau_x=\inf\{n\ge1:x_d+S_d(n)\le0\}$;
 \item $V(x)$ depends on $x_d$ only and is proportional to the renewal function of ladder heights of 
       the random walk $\{S_d(n)\}$.
\end{itemize}
In other words, the exit problem from $K$ is actually a one-dimensional problem. This allows is to use existing results for one-dimensional walks. 

The proof of Theorem \ref{thm:half-space} is based on the following simple generalization of known results for cones.
\begin{lemma}
\label{lem:half-space}
Assume that $\mathbf{E}\vert X\vert^{2+\delta}<\infty$. Then, uniformly in $x\in K$ with $x_d=o(\sqrt{n})$,
\begin{enumerate}[label=($\alph{*}$),ref=($\alph{*}$)]
     \item\label{lem:it-a}$\mathbf{P}(\tau_x>n)\sim \varkappa V(x)n^{-1/2}$;
     \item\label{lem:it-b}$(\frac{x+S([nt])}{\sqrt{n}})_{t\in[0,1]}$ conditioned on $\{\tau_x>n\}$ converges
             weakly to the Brownian meander in $K$;
     \item\label{lem:it-c}$\sup_{y\in K}\Big\vert n^{1/2+d/2}\mathbf{P}(x+S(n)=y;\tau_x>n)-cV(x)\frac{y_d}{\sqrt{n}}e^{-\vert y-x\vert^2/2n}\Big\vert\to0$.         
\end{enumerate}
\end{lemma}
\begin{proof}
The first statement is the well-known result for one-dimensional random walks, see \cite[Cor.~3]{D12}. The second and third 
statements for fixed starting points $x$ have been proved in \cite{DuWa-20} and in \cite{DeWa-15}, respectively.
To consider the case of growing $x_d$, one has to make only one change: Lemma 24 from \cite{DeWa-15} should
be replaced by the estimate
\begin{equation*}
\lim_{n\to\infty}\frac{1}{V(x)}
\mathbf{E}\left[\vert x+S(\nu_n)\vert;\tau_x>\nu_n,\vert x+S(\nu_n)\vert >\theta_n\sqrt{n},\nu_n\le n^{1-\varepsilon}\right]=0
\end{equation*}
uniformly in $x_d\le \theta_n\sqrt{n}/2.$ If $x_d\ge n^{1/2-\varepsilon}$ then $\nu_n=0$ and the 
expectation equals zero. If $x_d\le n^{1/2-\varepsilon}$ then one repeats the proof of  
\cite[Lem.~24]{DeWa-15} with $p$ replaced by $1$ and uses the part \ref{lem:it-a} of the lemma to obtain an estimate
for the sum $\sum_{j\le n^{1-\varepsilon}}\mathbf{P}(\tau_x>j-1)$ uniform in $x_d$. (In \cite{DeWa-15}, the Markov
inequality has been used, since one does not have the statement \ref{lem:it-a} in general cones.)
\end{proof}
\begin{lemma}
\label{lem:half-space-2}
Uniformly in $y$ with $y_d=o(\sqrt{n})$,
\begin{equation*}
     \mathbf{P}(x+S(n)=y,\tau_x>n)\sim c\frac{V(x)V'(y)}{n^{1+d/2}}e^{-\vert y\vert^2/2n}.
\end{equation*}
\end{lemma}
\begin{proof}
Set $m = \lfloor\frac{n}{2}\rfloor$ and write 
\begin{align*}
\pr(x+&S(n) = y,\tau_x>n) \\
&= \sum_{z\in K} \pr(x+S(n-m) = z,\tau_x>n-m)\pr(z+S(m) = y,\tau_z>m)\\
&=\sum_{z\in K} \pr(x+S(n-m) = z,\tau_x>n-m)\pr(y+S'(m) = z,\tau'_y>m),
\end{align*}
where we recall that $S'=-S$ is the reverse random walk and
\begin{equation*}
     \tau'_y:=\inf\{n\ge 1: y+S'(n)\notin K\}.
\end{equation*}
Applying part \ref{lem:it-c} of Lemma \ref{lem:half-space} to the random walk $\{S'(n)\}$,
we obtain
\begin{align*}
\pr(x+&S(n) = y,\tau_x>n)\\
&=\frac{cV'(y)}{m^{1+d/2}}\sum_{z\in K}z_de^{-\vert z-y\vert^2/2m}\pr(x+S(n-m) = z,\tau_x>n-m)\\
&\hspace{1cm}+o\left(V'(y)m^{-1/2-d/2}\pr(\tau_x>n-m)\right).
\end{align*}
Using now Lemma~\ref{lem:half-space} \ref{lem:it-a}, we get
\begin{align*}
\pr(x+&S(n) = y,\tau_x>n)\\
&=\frac{cV'(y)V(x)}{m^{1/2+d/2}(n-m)^{1/2}}
\e_x\left[\frac{S_d(n-m)}{\sqrt{m}}e^{-\vert S(n-m)-y\vert^2/2m}\Big\vert\tau_x>n-m\right]\\
&\hspace{1cm}+o\left(\frac{V(x)V'(y)}{m^{-1/2-d/2}(n-m)^{1/2}}\right).
\end{align*}
It follows from part \ref{lem:it-b} of the previous lemma that
\begin{align*}
\e_x\left[\frac{S_d(n-m)}{\sqrt{m}}e^{-\vert S(n-m)-y\vert^2/2m}\Big\vert \tau_x>n-m\right]\sim \e\left[\left(M_{K,d}(1)\right)e^{-\vert M_K-y/\sqrt{m}\vert^2/2}\right],
\end{align*}
where $M_K(t)=(M_{K,1}(t),M_{K,2}(t),\ldots,M_{K,d}(t))$ is the meander in $K$.

Since $K={\bf R}^{d-1}\times{\bf R}_+$, all coordinates of $M_K$ are independent.
Furthermore, $M_{K,1}(t),\ldots,M_{K,d-1}(t)$ are Brownian motions and $M_{K,d}(t)$ is the 
one-dimensional Brownian meander. Combining these observations with $y_d=o(\sqrt{n})$, we
conclude that 
\begin{align*}
\e\left[\left(M_{K,d}(1)\right)e^{-\vert M_K-y/\sqrt{m}\vert^2/2}\right]
&\sim \e\left[M_{K,d}(1)e^{-M_{K,d}^2/2}\right]\prod_{i=1}^{d-1}\e\left[e^{-(M_{K,i}(1)-y_i/\sqrt{m})^2/2}\right]\\
&= C\prod_{i=1}^{d-1} e^{-y_i^2/4m}\sim Ce^{-\vert y\vert^2/2n}.
\end{align*}
This completes the proof.
\end{proof}

\begin{proof}[Proof of Theorem~\ref{thm:half-space}]
If $y$ is such that $y_d\ge\alpha \vert y\vert$ for some $\alpha>0$ then it suffices to repeat the proof of 
Theorem~\ref{thm:deep}. We thus consider the boundary case $y_d=o(\vert y\vert)$. Using Lemma~\ref{lem:half-space-2}, one easily obtains
\begin{equation*}
\lim_{\varepsilon\to0}\lim_{\vert y\vert\to\infty}\frac{\vert y\vert^d}{V(x)v'(y)}S_2(x,y,\varepsilon)=c.
\end{equation*}
It follows that 
\begin{align*}
\lim_{\varepsilon\to0}\lim_{\vert y\vert\to\infty}\frac{\vert y\vert^d}{V(x)v'(y)}S_2(x,y,\varepsilon) &= c\lim_{\varepsilon\to0}\lim_{\vert y\vert\to\infty}\sum_{n\ge \epsilon\vert y\vert^2}\vert y\vert^dn^{-1-\frac{d}{2}}e^{-\frac{\vert y\vert^2}{2n}}
\\&= c\int_0^{\infty}v^{-1-{d}/{2}}e^{-\frac{1}{2v}}dv,
\end{align*}
and the last integral is finite. 
It follows that the theorem will be proven if we show that
\begin{equation}
\label{half-space.1}
\lim_{\varepsilon\to0}\lim_{\vert y\vert\to\infty}\frac{\vert y\vert^d}{V(x)V'(y)}S_1(x,y,\varepsilon)=0.
\end{equation}
Using an appropriate rotation, we can reduce everything to the case $y_k=o(\vert y\vert)$ for any $k=2,\ldots, d-1$
and $y_1\sim \vert y\vert$. This also implies $y_d = o(\vert y\vert)$.

We first split the probability $\mathbf{P}(x+S(n)=y,\tau_x>n)$ into two parts:
\begin{equation*}
      \mathbf{P}(x+S(n)=y,\tau_x>n,\max_{k\le n}\vert X_1(k)\vert\le\gamma y_1)\\
+\mathbf{P}(x+S(n)=y,\tau_x>n,\max_{k\le n}\vert X_1(k)\vert>\gamma y_1),
\end{equation*}
where $\gamma\in(0,1)$. Introduce the stopping time
\begin{equation*}
\sigma_\gamma:=\inf\{k\ge1:\vert X_1(k)\vert>\gamma y_1\}.
\end{equation*}
Then, by the Markov property,
\begin{align*}
&\mathbf{P}(x+S(n)=y,\tau_x>n,\max_{k\le n}\vert X_1(k)\vert>\gamma y_1)\\
&\hspace{1cm}=\sum_{k=1}^n \mathbf{P}(x+S(n)=y,\tau_x>n,\sigma_\gamma=k)\\
&\hspace{1cm}\le\sum_{k=1}^n\mathbf{P}(\tau_x>k-1)\mathbf{P}(\vert X_1\vert>\gamma y_1)\max_z\mathbf{P}(S(n-k)=z).
\end{align*}
Using now the bounds $\mathbf{P}(\tau_x>k)\le CV(x)k^{-1/2}$ and $\max_z\mathbf{P}(S(k)=z)\le Ck^{-d/2}$,
we obtain
\begin{align*}
&\mathbf{P}(x+S(n)=y,\tau_x>n,\max_{k\le n}\vert X_1(k)\vert>\gamma y_1)\\
&\hspace{1cm}\le CV(x)\mathbf{P}(\vert X_1\vert>\gamma y_1)\sum_{k=1}^n\frac{1}{\sqrt{k}}\frac{1}{(n-k+1)^{d/2}}\\
&\hspace{1cm}\le C V(x)\mathbf{P}(\vert X_1\vert>\gamma y_1)\frac{(\log n)^{\mathbf{1}_{d=2}}}{\sqrt{n}}.
\end{align*}

Here, in the last step we have splited the sum $\sum_{k=1}^n\frac{1}{\sqrt{k}}\frac{1}{(n-k+1)^{d/2}}$ into $\sum_{k=1}^\frac{n}{2}$ and $\sum_{k=\frac{n}{2}}^n$ and used elementary inequalities.

This implies that
\begin{equation*}
\sum_{n=1}^{\varepsilon\vert y\vert^2}\mathbf{P}(x+S(n)=y,\tau_x>n,\max_{k\le n}\vert X_1(k)\vert>\gamma y_1)\\
\le C\sqrt{\varepsilon}V(x)\mathbf{P}(\vert X_1\vert>\gamma y_1)\vert y\vert\left(\log \vert y\vert\right)^{\mathbf{1}_{d=2}}.
\end{equation*}
As a result, for all random walks satisfying
\begin{equation*}
\mathbf{E}\left[\vert X\vert^{d+1}\left(\log \vert X\vert\right)^{\mathbf{1}_{d=2}}\right]<\infty,
\end{equation*}
we have
\begin{equation}
\label{half-space.2} 
\sum_{n=1}^{\varepsilon\vert y\vert^2}\mathbf{P}(x+S(n)=y,\tau_x>n,\max_{k\le n}\vert X_1(k)\vert>\gamma y_1)
=o\left(\frac{V(x)}{\vert y\vert^d}\right).
\end{equation}

In order to estimate $\mathbf{P}(x+S(n)=y,\tau_x>n,\max_{k\le n}\vert X_1(k)\vert\le\gamma y_1)$ 
we shall perform the following change of measure:
\begin{equation*}
\overline{\mathbf{P}}(X(k)\in dz)
=\frac{e^{hz_1}}{\varphi(h)}\mathbf{P}(X(k)\in dz; \vert X_1(k)\vert\le\gamma y_1),
\end{equation*}
where
\begin{equation*}
\varphi(h)=\mathbf{E}\left[e^{hX_1};\vert X_1\vert\le \gamma y_1\right].
\end{equation*}
Therefore,
\begin{equation}
\label{half-space.3}
     \mathbf{P}(x+S(n)=y,\tau_x>n,\max_{k\le n}\vert X_1(k)\vert\le\gamma y_1)=e^{-hy_1}\varphi^n(h)\overline{\mathbf{P}}(x+S(n)=y,\tau_x>n).
\end{equation}
According to \cite[Eq.~(21)]{FuNa-71},
\begin{align*}
&e^{-hy_1}\varphi^n(h)\\
&\le
\exp\left\{-hy_1+hn\mathbf{E}[X_1;\vert X_1\vert\le\gamma y_1]+\frac{e^{h\gamma y_1}-1-h\gamma y_1}{\gamma^2y_1^2}
n\mathbf{E}[X_1^2;\vert X_1\vert\le\gamma y_1]\right\}.
\end{align*}
Choosing
\begin{equation}
\label{half-space.4}
h=\frac{1}{\gamma y_1}\log\left(1+\frac{\gamma y_1^2}{n\mathbf{E}[X_1^2;\vert X_1\vert\le\gamma y_1]}\right)
\end{equation}
and noting that
\begin{equation*}
\left\vert \mathbf{E}[X_1;\vert X_1\vert\le\gamma y_1]\right\vert
=\left\vert\mathbf{E}[X_1;\vert X_1\vert>\gamma y_1]\right\vert
\le\frac{1}{\gamma y_1}\mathbf{E}[X_1^2]=\frac{1}{\gamma y_1},
\end{equation*}
we conclude that uniformly for $n\le \gamma\vert y\vert^2$, it holds
\begin{equation*}
e^{-hy_1}\varphi^n(h)\le \left(\frac{en}{\gamma y_1^2}\right)^{1/\gamma}.
\end{equation*}
Plugging this into \eqref{half-space.3}, we obtain that uniformly for $n\le \gamma\vert y\vert^2$,
\begin{multline}
\label{eq:half-space.5}
\mathbf{P}\left(x+S(n)=y,\tau_x>n,\max_{k\le n}\vert X_1(k)\vert\le\gamma y_1\right)\\
\le C(\gamma)\left(\frac{n}{\vert y\vert^2}\right)^{1/\gamma}\overline{\mathbf{P}}(x+S(n)=y,\tau_x>n).
\end{multline}
According to \cite[Thm~6.2]{Esseen68}, there exists an absolute constant $C$ such that
\begin{equation*}
\sup_z\overline{\mathbf{P}}(S(n)=z)\le\frac{C}{n^{d/2}}\chi^{-d/2},
\end{equation*}
where
\begin{equation*}
\chi:=\sup_{u\ge 1}\frac{1}{u^2}\inf_{\vert t\vert=1}\overline{\mathbf{E}}\left[(t,X(1)-X(2));\vert X(1)-X(2)\vert\le u\right].
\end{equation*}
Since $h$ defined in \eqref{half-space.4} converges to zero as $\vert y\vert\to\infty$ uniformly in $n\le \gamma\vert y\vert^2$,
\begin{equation*}
\overline{\mathbf{E}}\left[(t,X(1)-X(2));\vert X(1)-X(2)\vert\le u\right]
\to\mathbf{E}\left[(t,X(1)-X(2));\vert X(1)-X(2)\vert\le u\right]
\end{equation*}
for every fixed $u$. Since $S(n)$ is truly $d$-dimensional under the original measure,
\begin{equation*}
     \inf_{\vert t\vert=1}\mathbf{E}\left[(t,X(1)-X(2));\vert X(1)-X(2)\vert\le u\right]>0
\end{equation*}
for all large values $u$.
As a result, there exists $\chi_0>0$ such that $\chi\ge\chi_0$ for all $\vert y\vert$ large enough and
all $n\le \gamma\vert y\vert^2$. Consequently,
\begin{equation}
\label{half-space.6}
\sup_z\overline{\mathbf{P}}(S(n)=z)\le \frac{C\chi_0^{-d/2}}{n^{d/2}}.
\end{equation}
Combining this bound with \eqref{eq:half-space.5}, we obtain for all $r\in(0,1)$ and $\gamma<2/d$,
\begin{multline*}
\sum_{n=1}^{\vert y\vert^{2-r}} \mathbf{P}\left(x+S(n)=y,\tau_x>n,\max_{k\le n}\vert X_1(k)\vert\le\gamma y_1\right)\\
\le C(\gamma)\chi_0^{-d/2}\vert y\vert^{-2/\gamma} \sum_{n=1}^{\vert y\vert^{2-r}}n^{1/\gamma-d/2}
\le C(\gamma)\chi_0^{-d/2}\vert y\vert^{-2/\gamma}\vert y\vert^{(2-r)(1/\gamma-d/2+1)},
\end{multline*}
for all $n\le \gamma\vert y\vert^2$.   
If we choose $\gamma$ so small that $r(1/\gamma-d/2+1)>2$, then
\begin{equation}
\label{half-space.7}
\sum_{n=1}^{\vert y\vert^{2-r}} \mathbf{P}\left(x+S(n)=y,\tau_x>n,\max_{k\le n}\vert X_1(k)\vert\le\gamma y_1\right)
=o\left(\frac{1}{\vert y\vert^d}\right).
\end{equation}

In the case $n\ge\vert y\vert^{2-r}$, we cannot ignore the condition $\tau_x>n$. By the Markov property at times
$n/3$ and $2n/3$ and by \eqref{half-space.6},
\begin{align*}
&\overline{\mathbf{P}}(x+S(n)=y,\tau_x>n)\\
&\le\sum_{z,z'}\overline{\mathbf{P}}(x+S(n/3)=z,\tau_x>n/3)\overline{\mathbf{P}}(z+S(n/3)=z')
\overline{\mathbf{P}}(z'+S(n/3)=y,\tau_{z'}>n/3)\\
&=\sum_{z,z'}\overline{\mathbf{P}}(x+S(n/3)=z,\tau_x>n/3)\overline{\mathbf{P}}(z+S(n/3)=z')
\overline{\mathbf{P}}(y+S'(n/3)=z',\tau_{y}>n/3)\\
&\le\frac{C}{n^{d/2}}\overline{\mathbf{P}}(\tau_x>n/3)\overline{\mathbf{P}}(\tau'_y>n/3).
\end{align*}
Therefore, it remains to show that, uniformly in $n\in[\vert y\vert^{2-r},\vert y\vert^2]$,
\begin{equation}
\label{half-space.8}
\overline{\mathbf{P}}(\tau_x>n/3)\le C\frac{1+x_d}{\sqrt{n}}.
\end{equation}
Indeed, from this estimate and from the corresponding estimate for the reverse walk we get
\begin{equation*}
\overline{\mathbf{P}}(x+S(n)=y,\tau_x>n)\le C\frac{(x_d+1)(y_d+1)}{n^{d/2+1}}.
\end{equation*}
With the help of \eqref{eq:half-space.5}, this implies that
\begin{equation*}
\sum_{n=\vert y\vert^{2-r}}^{\varepsilon\vert y\vert^2}\mathbf{P}\left(x+S(n)=y,\tau_x>n,\max_{k\le n}\vert X_1(k)\vert\le\gamma y_1\right) \le C\varepsilon^{1/\gamma-d/2}(x_d+1)(y_d+1)\vert y\vert^{-d}.
\end{equation*}
Combining this with \eqref{half-space.2} and \eqref{half-space.7}, we obtain \eqref{half-space.1}.

To derive \eqref{half-space.8}, we first estimate some moments of the random walk $S_d(n)$ under 
$\overline{\mathbf{P}}$. By definition of this probability measure,
\begin{equation*}
\overline{\mathbf{E}}[X_d]=\frac{1}{\varphi(h)}\mathbf{E}\left[X_de^{hX_1};\vert X_1\vert\le \gamma y_1\right].
\end{equation*}
For the expectation on the right-hand side, we have the representation
\begin{align*}
&\mathbf{E}\left[X_de^{hX_1};\vert X_1\vert\le \gamma y_1\right]\\
&\hspace{1cm}=\mathbf{E}\left[X_d;\vert X_1\vert\le \gamma y_1\right]+h\mathbf{E}\left[X_dX_1;\vert X_1\vert\le \gamma y_1\right]\\
&\hspace{3cm}+\mathbf{E}\left[X_d(e^{hX_1}-1-hX_1);\vert X_1\vert\le \gamma y_1\right]\\
&\hspace{1cm}=-\mathbf{E}\left[X_d;\vert X_1\vert> \gamma y_1\right]-h\mathbf{E}\left[X_dX_1;\vert X_1\vert> \gamma y_1\right]\\
&\hspace{3cm}+\mathbf{E}\left[X_d(e^{hX_1}-1-hX_1);\vert X_1\vert\le \gamma y_1\right].
\end{align*}
In the last step, we have used the equalities $\mathbf{E}[X_d]=\mathbf{E}[X_dX_1]=0$.
If
\begin{equation}
\label{half-space.9}
\mathbf{E}\vert X\vert^{3+\delta}<\infty,
\end{equation}
then by the Markov inequality,
\begin{equation*}
\mathbf{E}\left[X_d;\vert X_1\vert> \gamma y_1\right]+h\mathbf{E}\left[X_dX_1;\vert X_1\vert> \gamma y_1\right]
=o(y_1^{-2})=o(n^{-1}).
\end{equation*}
Therefore,
\begin{equation*}
\mathbf{E}\left[X_de^{hX_1};\vert X_1\vert\le \gamma y_1\right]=o(n^{-1})+
\mathbf{E}\left[X_d(e^{hX_1}-1-hX_1);\vert X_1\vert\le \gamma y_1\right].
\end{equation*}
It is obvious that $\vert e^x-1-x\vert\le \frac{x^2}{2}e^{\vert x\vert}$. Therefore,
\begin{align*}
\left\vert \mathbf{E}\left[X_d(e^{hX_1}-1-hX_1);\vert X_1\vert\le \gamma y_1\right]\right\vert
&\le \frac{h^2}{2}\mathbf{E}\left[\vert X_d\vert X_1^2e^{h\vert X_1\vert};\vert X_1\vert\le\gamma y_1\right]\\
&\le \frac{e}{2}h^2\mathbf{E}\vert X_d\vert X_1^2+h^2e^{h\gamma y_1}\mathbf{E}\left[\vert X_d\vert X_1^2;\vert X_1\vert>\frac{1}{h}\right]\\
&\le \frac{e}{2}h^2\mathbf{E}\vert X_d\vert X_1^2+h^{2+\delta}e^{h\gamma y_1}\mathbf{E}\vert X_d\vert \vert X_1\vert^{2+\delta}\\
&\le \frac{e}{2}h^2\mathbf{E}\vert X\vert^3 + h^{2+\delta}e^{h\gamma y_1}\mathbf{E}\vert X\vert^{3+\delta}.
\end{align*}
In the last step, we have used H\"older's inequality. 
It is immediate from the definition of $h$ that $h^2\le cn^{-1}$. Further, if $n\ge\vert y\vert^{2-r}$
with some $r<\frac{\delta}{2}$, then $h^{2+\delta}e^{h\gamma y_1}=o(n^{-1})$. From these estimates and from \eqref{half-space.9}, we obtain that uniformly in $n\in[\vert y\vert^{2-r},\vert y\vert^2]$, 
\begin{equation}
\label{half-space.10}
     \left\vert\mathbf{E}\left[X_de^{hX_1};\vert X_1\vert\le \gamma y_1\right]\right\vert\le\frac{c}{n}.
\end{equation}

By the same arguments,
\begin{align}
\label{half-space.11}
\nonumber
\varphi(h)&=\mathbf{E}\left[e^{hX_1};\vert X_1\vert\le\gamma y_1\right]\\
\nonumber
&=\mathbf{P}(\vert X_1\vert\le \gamma y_1)+h\mathbf{E}\left[X_1;\vert X_1\vert\le\gamma y_1\right]
+\mathbf{E}\left[e^{hX_1}-1-hX_1;\vert X_1\vert\le\gamma y_1\right]\\
\nonumber
&=1-\mathbf{P}(\vert X_1\vert> \gamma y_1)-h\mathbf{E}\left[X_1;\vert X_1\vert>\gamma y_1\right]
+\mathbf{E}\left[e^{hX_1}-1-hX_1;\vert X_1\vert\le\gamma y_1\right]\\
&=1+o(n^{-1}).
\end{align}
Combining this with \eqref{half-space.10}, we finally obtain
\begin{equation}
\label{half-space.12}
     \left\vert\overline{\mathbf{E}}X_d\right\vert\le\frac{c_1}{n}.
\end{equation}

We now turn to the second and third moments of $X_d$ under $\overline{\mathbf{P}}$.
Using \eqref{half-space.11} and the moment assumption, we have
\begin{align*}
\overline{\mathbf{E}}X_d^2&=\frac{1}{\varphi(h)}\mathbf{E}[X_d^2e^{hX_1};\vert X_1\vert\le\gamma y_1]
=(1+o(1))\mathbf{E}[X_d^2e^{hX_1};\vert X_1\vert\le\gamma y_1]\\
&=\mathbf{E}[X_d^2;\vert X_1\vert\le\gamma y_1]+o(1)+O\left(\mathbf{E}\left[X_d^2(e^{hX_1}-1);\vert X_1\vert\le\gamma y_1\right]\right)\\
&=1+o(1)+O\left(he^{h\gamma y_1}\right).
\end{align*}
Noting that $he^{h\gamma y_1}=o(1)$ for all $n\ge\vert y\vert^{2-r}$, we get
\begin{equation}
\label{half-space.13}
\overline{\mathbf{E}}X_d^2=1+o(1). 
\end{equation}
Similarly,
\begin{align*}
\overline{\mathbf{E}}\vert X_d\vert^3&=(1+o(1))\mathbf{E}[\vert X_d\vert^3e^{hX_1};\vert X_1\vert\le\gamma y_1]\\
&\le c\left(\mathbf{E}[\vert X_d\vert^3;\vert X_1\vert\le 1/h]+e^{h\gamma y_1}\mathbf{E}[\vert X_d\vert^3;\vert X_1\vert>1/h]\right)\\
&\le c\left(\mathbf{E}\vert X_d\vert^3+h^\delta e^{h\gamma y_1}\mathbf{E}\vert X_d\vert^{3+\gamma}\right).
\end{align*}
Using once again the fact that $h^\delta e^{h\gamma y_1}=o(1)$ for $n\ge\vert y\vert^{2-r}$, we arrive at
\begin{equation}
\label{half-space.14}
\overline{\mathbf{E}}\vert X_d\vert^3\le c_3.
\end{equation}

Now we can derive \eqref{half-space.8}. First, it follows from \eqref{half-space.12} that
\begin{equation*}
\overline{\mathbf{P}}(\tau_x>n/3)\le \overline{\mathbf{P}}(\tau^0_{x+c_1}>n/3),
\end{equation*}
where
\begin{equation*}
\tau_y^0:=\inf\{k\ge1:y+S_d^0(k)\le0\}\quad\text{and}\quad
S_d^0(k)=S_d(k)-k\overline{\mathbf{E}}X_d.
\end{equation*}
Applying \cite[Lem.~25]{DSW18} to the random walk $S_d^0$, we have
\begin{equation*}
\overline{\mathbf{P}}(\tau^0_y>k)
\le\frac{\overline{\mathbf{E}}[y+S_d^0(k);\tau_y^0>k]}{\overline{\mathbf{E}}[(y+S_d^0(k))^+]}.
\end{equation*}
Relations \eqref{half-space.13} and \eqref{half-space.14} allow the application of the central limit theorem
to the walk $S_d^0(k)$, which gives $\overline{\mathbf{E}}[(y+S_d^0(k))^+]\ge c\sqrt{k}$. Consequently,
\begin{equation*}
\overline{\mathbf{P}}(\tau^0_y>k)\le\frac{C}{\sqrt{k}}\overline{\mathbf{E}}[y+S_d^0(k);\tau_y^0>k].
\end{equation*}
Further, by the optional stopping theorem,
\begin{equation*}
\overline{\mathbf{E}}[y+S_d^0(k);\tau_y^0>k]=y-\overline{\mathbf{E}}[y+S_d^0(\tau_y^0);\tau_y^0\le k]
\le y-\overline{\mathbf{E}}[y+S_d^0(\tau_y^0)].
\end{equation*}
We now use inequality \cite[Eq.~(7)]{M73}, which states that there exists an absolute constant $A$ such that
\begin{equation*}
-\overline{\mathbf{E}}[y+S_d^0(\tau_y^0)]
\le A\frac{\overline{\mathbf{E}}\vert X_d\vert^3}{\overline{\mathbf{E}}X_d^2}.
\end{equation*}
Combining this with \eqref{half-space.13} and \eqref{half-space.14}, we finally get
\begin{equation*}
\overline{\mathbf{P}}(\tau^0_y>k)\le \frac{C(y+1)}{\sqrt{k}},
\end{equation*}
which implies \eqref{half-space.8}.
\end{proof}

\section{Boundary asymptotics of the Green function: the general case}
\label{sec:boundary}

The proof of Theorem \ref{thm:asymp_Green_boundary} consists in splitting the Green function $G_K(x,y)$ in \eqref{eq:Green_function_def} as a sum of two terms, the first (resp.\ second) one being given by the contribution in the large deviation (resp.\ asymptotic) regime. 

The main difficulty is to prove that the first term is actually dominated by the second one; in order to achieve this, we use a coupling of the random walk with a Brownian motion, with stronger bounds than the ones initially used in \cite{DeWa-15}. The drawback is the need of stronger moment assumptions on the increments, which is the main reason why the assumption \ref{H:moments} is used instead of the classical moment condition of \cite{DeWa-15}, namely $\e\vert X\vert^{r(p)}<\infty$ with $r(p)=p$ if $p>2$ and $r(p)=2+\delta$ for some $\delta>0$ if $p\leq 2$.

The arguments to show Theorem \ref{thm:asymp_Green_boundary} are very different in the $\mathcal C^2$ regular case, and two proofs are provided in this section.

\subsection*{Exact asymptotics with $\mathcal C^{2}$-regularity}
In this part, we will assume that the cone is $\mathcal C^{2}$. Before starting the proof of Theorem \ref{thm:asymp_Green_boundary}, we need to introduce some notation.
Let $\vert y\vert\to\infty$ in such a way that $\dist(y,\partial K)=o(\vert y\vert)$. 
Let $y_\perp\in\partial K$ be defined by the relation 
\begin{equation*}
     \dist(y,\partial K)=\vert y-y_\perp\vert.
\end{equation*}
Set $\sigma(y):=y_\perp/\vert y\vert\in\partial\Sigma$ and assume that $\sigma(y)$ converges as $\vert y\vert\to\infty$ to some $\bar\sigma\in \partial\Sigma$.
Let $H_y$ denote a tangent hyperplane at point $y_\perp.$ Let $P_n$ be the distribution
of the linear interpolation of $t\mapsto (y+S(nt))/\sqrt{n}$ conditioned to stay in the half-space $K_y$ containing the cone $K$ and having
boundary $H_y$. Then $P_n\to P$ weakly on $\mathcal C([0,1])$.
Denote
\begin{equation*}
A_n:=\{f\in \mathcal C([0,1]): f(k/n)\in K\text{ for all }1\le k\le n\}.
\end{equation*}
Then 
\begin{equation*}
\liminf A_n\supseteq\{f\in \mathcal C([0,1]): f(t)\in K\text{ for all }t\in(0,1]\}
\end{equation*}
and
\begin{equation*}
\limsup \overline{A}_n\subseteq\{f\in \mathcal C([0,1]): f(t)\in \overline{K}\text{ for all }t\in(0,1]\},
\end{equation*}
where $\overline{A}$ denote the closure of $A$.

Denote for every fixed $n$ by $[0,1]\ni t\mapsto S(nt)$ the linear interpolation of $\{S(k)\}_{k\le n}$. The conditions to apply  \cite[Thm~2.3]{Du-78} are met. This leads to an invariance principle: $[0,1]\ni r\mapsto\frac{y+S(nr)}{\sqrt{n}}$ converges weakly as $\frac{n}{\vert y\vert^2}\to t$ to the Brownian meander $(B_r)_{r\le 1}$ inside the cone $K$ started at $\frac{\sigma}{\sqrt{t}}$. In particular, with $T_y:=\inf\{n\ge1: y+S(n)\notin K_y\}$,
\begin{equation}
\label{1.step.1}
     \mathbf{P}\left(\frac{y+S(n)}{\sqrt{n}}\in B\Big\vert T_y>n\right)\sim Q_{\sigma,t}(B)=\int_B q_{\sigma,t}(z)dz,\quad
\frac{n}{\vert y\vert^2}\to t,
\end{equation}
where $q_{\sigma,t}(z)$ is the density of the Brownian meander in $K$, started at $\frac{\sigma}{\sqrt{t}}$ and evaluated at time $1$. 
Theorem 2.3 in \cite{Du-78} also leads to 
\begin{equation}
\label{1.step.2}
\mathbf{P}(\tau_y>n\vert T_y>n)\to c_{\sigma,t}.
\end{equation}
The relations \eqref{1.step.1} and \eqref{1.step.2} imply that
\begin{equation}
\label{V_low_bound}
V(y)\ge c\vert y\vert^{p-1}(1+\dist(y,\partial K)).
\end{equation}
Indeed, by the harmonicity of $V$, one has for all $n\ge1$,
\begin{equation*}
     V(y)=\mathbf{E}[V(y+S(n));\tau_y>n].
\end{equation*}
Fix now some $\epsilon>0$ and note that choosing $n=\lfloor\vert y\vert^2\rfloor$,
it follows that $V(z)\sim u(z)$ uniformly as $z\to\infty$ as long as the distance of $z$ to $\partial K$ is at least $\epsilon\vert z\vert$, see \cite[Lem.~13]{DeWa-15}.
We obtain, as $\vert y\vert\to\infty$ and $\epsilon\rightarrow 0$,
\begin{equation*}
     V(y)\ge \mathbf{P}(T_y>\lfloor\vert y\vert\rfloor^2)c_{\sigma,1}\vert y\vert^p\int_K u(z)q_{\sigma,1}(z)dz.
\end{equation*}
Due to results for the one-dimensional random walks, we arrive at
\begin{equation*}
     \mathbf{P}(T_y>\lfloor\vert y\vert\rfloor^2)\ge c\frac{1+\dist(y,\partial K)}{\vert y\vert},
\end{equation*}
which establishes \eqref{V_low_bound}.

Before proving Theorem \ref{thm:asymp_Green_boundary}, we record an auxiliary estimate needed in its proof.

\begin{lemma}\label{thm:auxmeander}
Define
\begin{equation*}
\phi_\sigma(t)=c_{\sigma,t}\int_K u(z)e^{-\frac{\vert z\vert^2}{2}}q_{\sigma,t}(z)dz.
\end{equation*}
Then there exists some $c>0$ such that as $t\to 0$, $\phi_\sigma(t)=o(e^{-c/t})$.
\end{lemma}

\begin{proof}
First, due to the invariance principle for the half-space, it holds 
\begin{equation*}
     c_{\sigma,t} = \pr_{\sigma}(\tau^{\me}>t) = \pr_{\sigma/\sqrt{t}}(\tau^{\me}>1),
\end{equation*}
where $\tau^{\me}: = \inf\{t>0:M^{\sigma}(t)\not\in K_y\}$. Here $M^{\sigma}(t)$ is a Brownian meander
in $K_{y}$, whereas we will denote the Brownian meander in $K$ by $M_K^{\sigma}(t)$.
Since $\vert \sigma\vert = 1$ and $K$ is contained in $K_y$, it is clear that $c_{\sigma,t}\ra 1$ as $t\ra 0$. 

Then we have
\begin{equation*}
     \phi_{\sigma}(t)\le C\e_{\sigma/\sqrt{t}}\left[u(M_K^{\sigma}(1))e^{-\frac{\vert M_K^{\sigma}(1)\vert^2}{2}}\right]\le C\e_{\sigma/\sqrt{t}}\left[u(M^{\sigma}(1))e^{-\frac{\vert M^{\sigma}(1)\vert^2}{2}}\right].
\end{equation*}
The second inequality can be easily justified using the invariance principles for meanders in $K$ and
$K_y$ as well as the fact that $c_{\sigma,t}\ra 1$ is bounded away from zero. It follows that
\begin{equation*}
     \phi_{\sigma}(t)\le C\e_{\sigma/\sqrt{t}}\left[e^{-\frac{\vert M^{\sigma}(1)\vert^2}{4}}\right].
\end{equation*}
Due to rotational invariance of Brownian motion, the expectation above doesn't depend on $\sigma$,
so that we can choose $\sigma = (1,0,\ldots,0)$ and $K_y = \R^{d-1}\times \R_+$. The first
$d-1$ coordinates become independent Brownian motions, whereas the last one is a
$1$-dimensional Brownian meander (see \cite{DIM77} for its density). This finishes the proof.
\end{proof}



\begin{proof}[Proof of Theorem \ref{thm:asymp_Green_boundary} when $K$ is $\mathcal C^{2}$] 
To estimate the contribution coming from large values of $n$,
one does not need the limit theorems from the previous paragraph: quite rough estimates turn out to be sufficient.

Set $m=\lfloor n/2\rfloor$. Then, applying the Markov property at time $m$
and inverting the time in the second part of the path, we obtain
\begin{align*}
&\mathbf{P}(x+S(n)=y,\tau_x>n)\\
&\hspace{1cm}=\sum_{z\in K}\mathbf{P}(x+S(m)=z,\tau_x>m)\mathbf{P}(y+S'(n-m)=z,\tau_y'>n-m)\\
&\hspace{1cm}\le\max_{z\in K}\mathbf{P}(x+S(m)=z,\tau_x>m)\mathbf{P}(\tau_y'>n-m).
\end{align*}
By \cite[Thm~5]{DeWa-15},
\begin{equation*}
\max_{z\in K}\mathbf{P}(x+S(m)=z,\tau_x>m)\le C\frac{V(x)}{m^{p/2+d/2}}.
\end{equation*}
Furthermore, due to results for the one-dimensional walks (see for example \cite[Lem.~3]{DW16}),
\begin{equation}
\label{3.step.1}
\mathbf{P}(\tau_y'>n-m)\le \mathbf{P}(T'_y>n-m)\le C\frac{1+\dist(y,\partial K)}{\sqrt{n-m}}.
\end{equation}
Combining these estimates, we obtain
\begin{equation*}
\mathbf{P}(x+S(n)=y)\le CV(x)(1+\dist(y,\partial K))n^{-(p+d+1)/2}.
\end{equation*}
Consequently, for $A\ge2$ and $\vert y\vert\ge1$,
\begin{align}
\label{3.step}
\nonumber
\sum_{n\ge A\vert y\vert^2}\mathbf{P}(x+S(n)=y)
&\le CV(x)(1+\dist(y,\partial K))\sum_{n\ge A\vert y\vert^2}n^{-(p+d+1)/2}\\
&\le CV(x)A^{-(p+d-1)/2}\frac{1+\dist(y,\partial K)}{\vert y\vert^{p+d-1}}.
\end{align}

We turn now to the middle part, namely, $n\in(\varepsilon\vert y\vert^2,A\vert y\vert^2)$.
Using again the Markov property at time $m=\lfloor n/2\rfloor$ and applying \cite[Thm~5]{DeWa-15},
we obtain
\begin{align*}
&\mathbf{P}(x+S(n)=y,\tau_x>n)\\
&=\sum_{z\in K}\mathbf{P}(x+S(m)=z,\tau_x>m)\mathbf{P}(y+S'(n-m)=z,\tau'_y>n-m)\\
&=\frac{\varkappa H_0V(x)}{m^{p/2+d/2}}
\sum_{z\in K}\left(u\left(\frac{z}{\sqrt{m}}\right)e^{-\frac{\vert z\vert^2}{2m}}+o(1)\right)\mathbf{P}(y+S'(n-m)=z,\tau'_y>n-m)\\
&=\frac{\varkappa H_0V(x)}{m^{p/2+d/2}}
\mathbf{E}\left[u\left(\frac{S'(n-m)}{\sqrt{m}}\right)e^{-\frac{\vert S'(n-m)\vert^2}{2m}};\tau'_y>n-m\right]
+o\left(\frac{\mathbf{P}(\tau'_y>n-m)}{m^{p/2+d/2}}\right).
\end{align*}
Taking into account \eqref{3.step.1}, we have
\begin{align*}
&\mathbf{P}(x+S(n)=y,\tau_x>n)\\
&=\frac{\varkappa H_0V(x)}{m^{p/2+d/2}}
\mathbf{E}\left[u\left(\frac{S'(n-m)}{\sqrt{m}}\right)e^{-\frac{\vert S'(n-m)\vert^2}{2m}};\tau'_y>n-m\right]
+o\left(\frac{1+\dist(y, \partial K)}{n^{(p+d+1)/2}}\right). 
\end{align*}
Next, it follows from \eqref{1.step.1} and \eqref{1.step.2} that if $\frac{n}{\vert y^2\vert}\sim t$, then
\begin{align*}
\mathbf{E}\left[u\left(\frac{S'(n-m)}{\sqrt{m}}\right)e^{-\frac{\vert S'(n-m)\vert^2}{2m}};\tau'_y>n-m\right]
\sim \mathbf{P}(T'_y>n-m) \phi_\sigma(t/2).
\end{align*}
Since $T'_y$ is an exit time from a half-space,
\begin{equation*}
\mathbf{P}(T'_y>k)\sim v'(y)k^{-1/2},
\end{equation*}
where $v'(y)$ is the positive harmonic function for $S'$ killed at leaving the half-space $K_{\sigma}$.
As a result,
\begin{equation*}
\mathbf{P}(x+S(n)=y,\tau_x>n)
=C_0 \frac{V(x)v'(y)}{n^{(p+d+1)/2}}\phi_\sigma\left(\frac{n}{\vert y\vert^2}\right)
+o\left(\frac{1+\dist(y, \partial K)}{n^{(p+d+1)/2}}\right),
\end{equation*}
where
\begin{equation*}
C_0:=\varkappa H_0 2^{(p+d+1)/2}.
\end{equation*}
This representation implies that
\begin{align*}
&\sum_{\varepsilon\vert y\vert^2}^{A\vert y\vert^2}\mathbf{P}(x+S(n)=y,\tau_x>n)\\
&\hspace{1cm}= C_0V(x)v'(y)\sum_{\varepsilon\vert y\vert^2}^{A\vert y\vert^2}n^{-(p+d+1)/2}\phi_\sigma\left(\frac{n}{2\vert y\vert^2}\right)
+o\left(\frac{1+\dist(y, \partial K)}{n^{(p+d-1)/2}}\right)\\
&\hspace{1cm}= C_0\frac{V(x)v'(y)}{\vert y\vert^{p+d-1}}\int_\varepsilon^A\phi_\sigma(t/2) t^{-(p+d+1)/2}dt
+o\left(\frac{1+\dist(y, \partial K)}{n^{(p+d-1)/2}}\right).
\end{align*}
Combining this with \eqref{3.step} and letting $A\to\infty$, one can easily obtain
\begin{equation*}
\lim_{\vert y\vert\to\infty}\frac{\vert y\vert^{p+d-1}}{V(x)v'(y)}S_2(x,y,\varepsilon)
=C_0\int_\varepsilon^\infty \phi_\sigma(t/2) t^{-(p+d+1)/2}dt.
\end{equation*}
From Lemma \ref{thm:auxmeander} it follows
\begin{equation}
\label{S2.boundary}
\lim_{\varepsilon\to0}\lim_{\vert y\vert \to\infty}\frac{\vert y\vert^{p+d-1}}{V(x)v'(y)}S_2(x,y,\varepsilon)
=C_0\int_0^\infty \phi_\sigma(t/2) t^{-(p+d+1)/2}dt.
\end{equation}

It remains to estimate $S_1(x,y,\varepsilon)$. We shall use the same strategy as in the proof of 
Theorem~\ref{thm:deep}, but instead of the Green function for the whole space we shall use the Green
function for the half-space $K_y$. More precisely,
\begin{align*}
&S_1(x,y,\varepsilon)=\sum_{n<\varepsilon\vert y^2\vert}\mathbf{P}(x+S(n)=y,\tau_x>n\ge\theta_y)\\
&=\sum_{n<\varepsilon\vert y^2 \vert}\sum_{k=1}^n\sum_{z\in B_{\delta,y}}
\mathbf{P}(x+S(n)=z,\tau_x>k=\theta_y)\mathbf{P}(z+S(n-k)=y,\tau_z>n-k)\\
&= \sum_{k<\varepsilon\vert y\vert^2}\sum_{z\in B_{\delta,y}}
\mathbf{P}(x+S(n)=z,\tau_x>k=\theta_y)\sum_{j<\varepsilon\vert y\vert^2-k}\mathbf{P}(z+S(j)=y,\tau_z>j)\\
&\le \sum_{k<\varepsilon\vert y\vert^2}\sum_{z\in B_{\delta,y}}
\mathbf{P}(x+S(n)=z,\tau_x>k=\theta_y)\sum_{j<\varepsilon\vert y\vert^2}\mathbf{P}(y+S'(j)=z,T_y'>j)\\
&=\mathbf{E}\left[G_{\varepsilon, y}(x+S(\theta_y));\tau_x>\theta_y,\theta_y\le\varepsilon\vert y\vert^2\right],
\end{align*}
where
\begin{equation*}
G_{\varepsilon, y}(z)=\sum_{j<\varepsilon\vert y\vert^2}\mathbf{P}(y+S'(j)=z,T_y'>j).
\end{equation*}
Applying Theorem \ref{thm:half-space} and \eqref{Green_large} to the random walk $S'(n)$, we obtain
\begin{equation*}
G_{\varepsilon, y}(z)\le
C\frac{v'(y)(1+\dist(z,H_y))}{1+\vert z-y\vert^d}\wedge1.
\end{equation*}
Therefore,
\begin{multline}
\label{S1_boundary.1}
S_1(x,y,\varepsilon)
\le C\mathbf{P}(\vert y-x-S(\theta_y)\vert\le\delta^2\vert y\vert,\tau_x>\theta_y,\theta_y\le\varepsilon\vert y\vert^2)\\
+C(\delta)\frac{v'(y)}{\vert y\vert^d}\mathbf{E}\left[(1+\dist(x+S(\theta_y),H_y);
\tau_x>\theta_y,\theta_y\le\varepsilon\vert y\vert^2\right].
\end{multline}
The first term has been estimated in \eqref{S1_estim_3} for random walks having finite moments of order $r_2(p):=p+d-1+(2-p)^+$:
\begin{equation}
\label{S1_boundary.2}
     \mathbf{P}(\vert y-x-S(\theta_y)\vert\le\delta^2\vert y\vert,\tau_x>\theta_y,\theta_y\le\varepsilon\vert y\vert^2)
=o(\vert y\vert^{-p-d+1}).
\end{equation}

In order to estimate the second term in \eqref{S1_boundary.1}, we shall perform again the change
of measure with the harmonic function $V$:
\begin{multline*}
\mathbf{E}\left[(1+\dist(x+S(\theta_y),H_y);\tau_x>\theta_y,\theta_y\le\varepsilon\vert y\vert^2\right]\\
=V(x)\mathbf{E}^{(V)}\left[\frac{1+\dist(x+S(\theta_y),H_y)}{V(x+S(\theta_y))};
\theta_y\le\varepsilon\vert y\vert^2\right].
\end{multline*}
Applying now \eqref{V_low_bound}, we obtain
\begin{equation*}
\mathbf{E}\left[(1+\dist(x+S(\theta_y),H_y);\tau_x>\theta_y,\theta_y\le\varepsilon\vert y\vert^2\right]
\le CV(x)\vert y\vert^{-p+1}\mathbf{P}^{(V)}(\theta_y\le\varepsilon\vert y\vert^2).
\end{equation*}
From this estimate and \eqref{S1_estim_4}, we conclude that
\begin{equation*}
\lim_{\varepsilon\to0}\lim_{\vert y\vert\to\infty}\vert y\vert^{p-1}
\mathbf{E}\left[(1+\dist(x+S(\theta_y),H_y);\tau_x>\theta_y,\theta_y\le\varepsilon\vert y\vert^2\right]=0.
\end{equation*}
Combining this estimate with \eqref{S1_boundary.1} and \eqref{S1_boundary.2} as well as \cite[Lem.~13]{DeWa-15}, we get
\begin{equation}
\label{S1.boundary}
\lim_{\varepsilon\to0}\lim_{\vert y\vert\to\infty}\vert y\vert^{p+d-1}S_1(x,y,\varepsilon)=0.
\end{equation}
Since $v'(y)$ is bounded from below by a positive number, \eqref{S1.boundary} and \eqref{S2.boundary}
yield the desired result for the case $\e[\vert X\vert^{r_2(p)}]<\infty$ due to classical results for the one-dimensional random walk.

Assume now that \ref{local_assump} holds. It is easy to see that the above proof that 
\begin{equation}
\lim_{\varepsilon\to0}\lim_{\vert y\vert\to\infty}\frac{\vert y\vert^{p+d-1}}{V(x)v'(y)}S_2(x,y,\varepsilon)
=C_0\int_0^\infty \phi_\sigma(t) t^{-(p+d+1)/2}dt,
\end{equation}
goes through again word for word. Therefore we focus on the asymptotics of $S_1(x,y,\varepsilon)$ in the following. 
With similar steps as above it holds 
\begin{align*}
&S_1(x,y,\varepsilon)\\
&\le C(\delta)v'(y)\e\left[\frac{1+\dist(x+S(\theta_y),H_y)}{1+\vert x+S(\theta_y)-y\vert^d},\vert y-x-S(\theta_y)\vert\le\delta^2\vert y\vert,\tau_x>\theta_y,\theta_y\le\varepsilon\vert y\vert^2\right]\\&+
C(\delta)\frac{v'(y)}{\vert y\vert^d}\mathbf{E}\left[(1+\dist(x+S(\theta_y),H_y);
\tau_x>\theta_y,\theta_y\le\varepsilon\vert y\vert^2\right].
\end{align*}
The second summand can be treated just as above with help of \eqref{V_low_bound} so that we need to show
\begin{multline*}
\e\left[\frac{1+\dist(x+S(\theta_y),H_y)}{1+\vert x+S(\theta_y)-y\vert^d},\vert y-x-S(\theta_y)\vert\le\delta^2\vert y\vert,\tau_x>\theta_y,\theta_y\le\varepsilon\vert y\vert^2\right]\\
= O(\vert y\vert^{-p-d+1}).
\end{multline*}
It holds
\begin{equation*}
1+\dist(x+S(\theta_y),H_y)\le 1+\vert S(\theta_y)-y\vert+\vert y-y_{\perp}\vert= o(\vert y\vert)+ \vert S(\theta_y)-y\vert.
\end{equation*}
To complete the proof we now show for $r = d-1$ and $r=d$, 
\begin{align*}
\label{eq:helpass4}
S_{2,r}(x,y,\varepsilon)&:=\e\left[\frac{1}{1+\vert x+S(\theta_y)-y\vert^{r}},\vert y-x-S(\theta_y)\vert\le\delta^2\vert y\vert,\tau_x>\theta_y,\theta_y\le\varepsilon\vert y\vert^2\right]\\
& = o(\vert y\vert^{-p-d+1}).
\end{align*}
With a similar calculation as in the proof of Theorem \ref{thm:deep} (using \eqref{local_1}), we obtain
\begin{align*}
&\mathbf{E}\left[\frac{1}{1+\vert y-x-S(\theta_y)\vert^{d-1}};\vert y-x-S(\theta_y)\vert\le\delta^2\vert y\vert,\tau_x>\theta_y,\theta_y\le\varepsilon\vert y\vert^2\right]\\
&\hspace{1cm}\le C(\delta)\vert y\vert^{-p-d+1}f(\delta(1-\delta)\vert y\vert)\mathbf{E}[\tau_x;\tau_x<\vert y\vert^2]\sum_{m=1}^{\delta^2\vert y\vert}\frac{m^{d-1}}{m^{d-1}}\\
&\hspace{1cm} \le C(\delta)\vert y\vert^{-p-d+2}f(\delta(1-\delta)\vert y\vert)\vert y\vert^{(2-p)^+}.
\end{align*}
Finally, 
\begin{align*}
&\mathbf{E}\left[\frac{1}{1+\vert y-x-S(\theta_y)\vert^{d}};\vert y-x-S(\theta_y)\vert\le\delta^2\vert y\vert,\tau_x>\theta_y,\theta_y\le\varepsilon\vert y\vert^2\right]\\
&\hspace{1cm}\le C(\delta)\vert y\vert^{-p-d+1}f(\delta(1-\delta)\vert y\vert)\mathbf{E}[\tau_x;\tau_x<\vert y\vert^2]\sum_{m=1}^{\delta^2\vert y\vert}\frac{m^{d-1}}{m^{d}}\\
&\hspace{1cm} \le C(\delta)\log(\vert y\vert)\vert y\vert^{-p-d+2}f(\delta(1-\delta)\vert y\vert)\vert y\vert^{(2-p)^+}.
\end{align*}
This finishes the proof of Theorem \ref{thm:asymp_Green_boundary} when $K$ is $\mathcal C^{2}$.
\end{proof}

\subsection*{Exact asymptotics in the general case}

We now turn to the general convex case, without assuming that the boundary is $\mathcal C^2$.
Recall from \eqref{eq:K_rho} the definition of
\begin{equation*}
     {K}_{\rho}:=\lbrace y\in{K}:\dist(y,\partial {K})\geq \vert y\vert^{1-\rho}\rbrace,
\end{equation*}
where $\rho$ is given in Theorem \ref{thm:deep}. Further, for $y\in{K}$, $\theta_{y}=\inf\{n\geq 0:y+S'(n)\in{K}_{\rho}\}$ was introduced in \eqref{eq:stopping_time_theta}.

\begin{proof}[Proof of Theorem \ref{thm:asymp_Green_boundary} in the general case] 
Split the Green function as
\begin{align*}
G_K(x,y)&=\sum_{n=1}^{\vert y-x\vert^{2-2\epsilon}-1}\pr(y+S'(n)=x,\tau'_{y}>n)\\
&\quad+\sum_{n=\vert y-x\vert^{2-2\epsilon}}^{\infty}\pr(y+S'(n)=x,\tau'_{y}>n,\theta_{y}\geq \vert y-x\vert^{2-2\epsilon})\\
&\quad+\sum_{n=\vert y-x\vert^{2-2\epsilon}}^{\infty}\pr(y+S'(n)=x,\tau'_{y}>n,\theta_{y}\leq \vert y-x\vert^{2-3\epsilon},\vert S'_{\theta_{y}}\vert \geq \vert y-x\vert^{1-\epsilon/\alpha})\\
&\quad+\sum_{n=\vert y-x\vert^{2-2\epsilon}}^{\infty}\pr(y+S'(n)=x,\tau'_{y}>n,\theta_{y}\leq \vert y-x\vert^{2-3\epsilon}\wedge \tau'_{y},\vert S'_{\theta_{y}}\vert \leq \vert y-x\vert^{1-\epsilon/\alpha})\\
&:=T_{1}+T_{2}+T_{3}+T_{4},
\end{align*}
and we shall study successively the terms $T_1$, $T_2$, $T_3$ and $T_4$.

\subsubsection*{Study of $T_1$ and $T_2$.} 

It follows from Proposition \ref{bound_local_far_away} that $T_{1}\leq C\vert y-x\vert^{-a}$, for some parameter $a>p+q+d-2+(2-p)^{+}$. In order to analyse $T_2$, we need the preliminary estimates \eqref{fastExitCepsilon} and \eqref{smallExitCepsilon} below. To that purpose, remark that $\theta_{y}=t'_{y,\rho}(\vert y\vert^{2-2\rho})$, see \eqref{eq:t_x}. Hence, noting that $t'_{y,\epsilon}(n)$ is increasing in $n$, we get with \cite[Lem.~14]{DeWa-15}
\begin{equation}
\label{fastExitCepsilon}
     \pr(\theta_{y}>n^{1-\epsilon}, \tau'_{y}> n)\leq \pr(t'_{y,\epsilon}(n)>n^{1-\epsilon}, \tau'_{y}> n) \leq C\exp(n^{-\epsilon})
\end{equation}
for $n\geq \vert y\vert^{2-2\epsilon}$. Applying Lemma \ref{Lemma_24_revisited} to the stopping time $\theta_{y}$ and using the moment condition $\e\vert X\vert^{r(p)}<\infty$, we obtain that there exist $C>0$ and $\alpha>0$ such that 
\begin{equation}
\label{smallExitCepsilon}
     \pr\big(\vert S'_{\theta_{y}}\vert \geq \vert y-x\vert^{1-\epsilon/\alpha},\theta_{y}\leq \vert y-x\vert^{2-\epsilon}, \tau'_{y}> \vert y-x\vert^{2-\epsilon}\big)\leq C\vert y-x\vert^{-s},
\end{equation}
with $s>(2-2\epsilon)(p+q+d-4+(2-p)^{+})/2$.

Let us now write 
\begin{multline*}
T_{2}=\sum_{n\geq \vert y-x\vert^{2-2\epsilon}}\pr(x+S(n)=y,\tau_{x}>n,\theta_{y}\geq n^{1-\epsilon})\\
+\sum_{n\geq \vert y-x\vert^{2-2\epsilon}}\pr(x+S(n)=y,\tau_{x}>n,\vert y-x\vert^{2-3\epsilon}\leq  \theta_{y}\leq n^{1-\epsilon}).
\end{multline*}
By \eqref{fastExitCepsilon}, the first term is bounded by $\sum_{n\geq \vert y-x\vert^{2-2\epsilon}}C\exp(-n^{\epsilon})\leq C\exp(-\vert y-x\vert^{\epsilon'})$ for some $C>0$ and some $0<\epsilon'<\epsilon$. Moreover, by \eqref{bound_local_probability_in_cone} and \eqref{fastExitCepsilon},
\begin{align*}
&\sum_{n\geq \vert y-x\vert^{2-2\epsilon}}\pr(x+S(n)=y,\tau_{x}>n,\vert y-x\vert^{2-3\epsilon}\leq  \theta_{y}\leq n^{1-\epsilon})\\
&=\sum_{n\geq \vert y-x\vert^{2-2\epsilon}}\e[x+S(n-\theta_{y})=y+S'(\theta_{y});\tau_{x}>n-\theta_{y},\tau'_{y}>\theta_{y},\vert y-x\vert^{2-3\epsilon}\leq  \theta_{y}\leq n^{1-\epsilon}]\\
&\leq\sum_{n\geq \vert y-x\vert^{2-2\epsilon}}C(n-n^{1-\epsilon})^{-d/2-p/2}\pr(\theta_{y}>\vert y-x\vert^{2-3\epsilon},\tau_{y} \geq \vert y-x\vert^{2-3\epsilon})\\
&\leq  C\exp(-\vert y-x\vert^{\epsilon(2-3\epsilon)}),
\end{align*}
so that, finally,
\begin{equation*}
     T_{2}\leq  C\exp(-\vert y-x\vert^{\epsilon'}),
\end{equation*}
for some constant $C>0$ and $0<\epsilon'<\epsilon$.

\subsubsection*{Study of $T_3$.}
By \eqref{bound_local_probability_in_cone}, we have for $n\geq \vert x-y\vert^{2-2\epsilon}$ and $y$ large enough
\begin{align*}
\pr(y+S'(n)=x&,\tau'_{y}>n,\theta_{y}\leq \vert y-x\vert^{2-3\epsilon},\vert S'_{\theta_{y}}\vert \geq \vert y-x\vert^{1-\epsilon/\alpha})\\
&\leq \e[(n-\theta_{y})^{-p/2-d/2};\tau'_{y}>\theta_{y},\theta_{y}\leq \vert y-x\vert^{2-3\epsilon},\vert S'_{\theta_{y}}\vert \geq \vert y-x\vert^{1-\epsilon/\alpha}]\\
&\leq Cn^{-p/2-d/2}\pr(\tau'_{y}>\theta_{y},\theta_{y}\leq \vert y-x\vert^{2-3\epsilon},\vert S'_{\theta_{y}}\vert \geq \vert y-x\vert^{1-\epsilon/\alpha})\\
&\leq  Cn^{-p/2-d/2}\vert y-x\vert^{s}.
\end{align*}
Hence,
\begin{equation*}T_{3}\leq C\sum_{n=\vert x-y\vert^{2-2\epsilon}}^{\infty} n^{-p/2-d/2}\vert y-x\vert^{s}\leq \vert y-x\vert^{-(2-2\epsilon)(p/2+d/2-1)-s}.\end{equation*}
By the definition of $s$ given in \eqref{smallExitCepsilon},
\begin{align*}
     (2-2\epsilon)(p/2+d/2-1)+s&> (p+d-2)+(p+q+d-4+2(1-p/2)^{+})+f(\epsilon)\\
     &=p+q+d-2+(p+d-4+2(1-p/2)^{+})+g(\epsilon),
\end{align*}
with $g$ linear. Since $p+d-4+2(1-p/2)^{+}\geq 0$ for all $p\geq 1$ and $d\geq 2$, 
\begin{equation*}
     T_{3}=o(\vert x-y\vert^{-b}),
\end{equation*}
with $b>p+q+d-2$ for $\epsilon$ small enough.

\subsubsection*{Study of $T_4$.}
By Theorem \ref{thm:deep}, we have
\begin{align*}
T_{4}&\geq\e[G_{\vert x-y\vert^{2-2\epsilon}}(x,y+S'(\theta_{y}));\tau'_{y}>\theta_{y},\theta_{y}\leq \vert y-x\vert^{2-3\epsilon},\vert S(\theta_{y})\vert \leq \vert y-x\vert^{1-\epsilon/\alpha}]\\&\geq C\e[u( y+S'(\theta_{y}))\vert y+S'(\theta_{y})\vert^{-2p-d+2};\tau'_{y}>\theta_{y},\theta_{y}\leq \vert y-x\vert^{2-3\epsilon},\vert S(\theta_{y})\vert \leq \vert y-x\vert^{1-\epsilon/\alpha}]\\
&\geq C\vert y\vert^{-2p-d+2}\vert y\vert^{p-q''+q'+O(\epsilon)}\pr(\tau'_{y}>\theta_{y},\theta_{y}\leq \vert y-x\vert^{2-3\epsilon},\vert S(\theta_{y})\vert \leq \vert y-x\vert^{1-\epsilon/\alpha}),
\end{align*}
for some $q''>q'>q$ small enough, where we have used the fact that $y+S'(\theta_{y})\in{K}_{\epsilon}$, $\vert S(\theta_{y})\vert \leq \vert y-x\vert^{1-\epsilon/\alpha}$ and Lemma \ref{lower_bound_u_tangent} to give a lower bound on $u(y+S'(\theta_{y}))$. Hence,
\begin{equation*}T_{4}\geq \vert y\vert^{-p-q''+q'-d+2+O(\epsilon)}(\pr(\tau'_{y}>\vert y-x\vert^{2-3\epsilon})-C\exp(-\vert y\vert^{\epsilon})-K \vert y-x\vert^{s}).\end{equation*}
By Lemma \ref{lower_bound_survival_tangent}, $\pr(\tau'_{y}>\vert y-x\vert^{2-3\epsilon}) \geq C \vert y-x\vert^{-q'/2(2-3\epsilon)}$ and $s>q'/2$ for $\epsilon$ and $q'$ small enough, which yields 
\begin{equation*}T_{4}\geq c\vert y\vert^{-p-q'-d+2+O(\epsilon)}.\end{equation*}
Hence, for $\epsilon$ and $q'$ small enough,
\begin{equation*}T_{1}+T_{2}+T_{3}=o(T_{4}).\end{equation*}
Moreover, by Theorem \ref{thm:deep},
\begin{multline*}
\e[G_{K}(x,y+S'(\theta_{y}));\tau'_{y}>\theta_{y},\theta_{y}\leq \vert y-x\vert^{2-3\epsilon},\vert S(\theta_{y})\vert \leq \vert y-x\vert^{1-\epsilon/\alpha}]\\
 \sim V(x)\vert y\vert^{-2p-q+2}\e[u(y+S'(\theta_{y})),\tau'_{y}>\theta_{y},\theta_{y}\leq \vert y-x\vert^{2-3\epsilon},\vert S(\theta_{y})\vert \leq \vert y-x\vert^{1-\epsilon/\alpha}],
\end{multline*}
as $y$ goes to infinity. Since we also have 
\begin{multline*}
 \e[u(y+S'(\theta_{y}));\tau'_{y}>\theta_{y},(\theta_{y}\geq \vert y-x\vert^{2-3\epsilon})\cup(\vert S(\theta_{y})\vert \geq \vert y-x\vert^{1-\epsilon/\alpha})]\\
=o(\e[u(y+S'(\theta_{y}));\tau'_{y}>\theta_{y},\theta_{y}\leq \vert y-x\vert^{2-3\epsilon},\vert S(\theta_{y})\vert \leq \vert y-x\vert^{1-\epsilon/\alpha}])
\end{multline*}
for the same reasons as before, the result is deduced.
\end{proof}
The uniqueness of the harmonic function is then a straightforward deduction of the latter theorem together with Martin boundary theory.
\begin{corollary}
The Martin boundary of $S$ killed on the boundary of ${K}$ is reduced to a singleton, and there exists a unique harmonic function (up to multiplication by a constant).
\end{corollary}
\begin{proof}
Let $x_{0},x\in{K}$ and let $(y_{n})$ be a sequence in ${K}$ going to infinity. Then, by Theorems \ref{thm:deep} and \ref{thm:asymp_Green_boundary}, as $n\to \infty$,
\begin{equation*}
     \frac{G_K(x,y_{n})}{G_K(x_{0},y_{n})}\to\frac{V(x)}{V(x_{0})}.
\end{equation*}
The Martin boundary is thus reduced to a singleton.
\end{proof}

\section{Optimality of the moment conditions}
\label{sec:optimal}

In this section, we prove that the assumptions of Theorems \ref{thm:deep} and \ref{thm:asymp_Green_boundary} are optimal.

Uchiyama~\cite{Uc-98} has shown, see Theorem 2 there, that if $d\ge 5$ and
$\mathbf{E}\vert X\vert^{d-2}<\infty$, then
\begin{equation*}
     G_{{\bf R}^d}(0,z)\sim\frac{c}{\vert z\vert^{d-2}}
\end{equation*}
as $\vert z\vert\to\infty$. The same asymptotics is valid when $d=4$ or $d=3$, provided that respectively
$\mathbf{E}\vert X\vert^2\log\vert X\vert<\infty$ or $\mathbf{E}\vert X\vert^2<\infty$.

Uchiyama mentions also that this moment condition is optimal:
for any $\varepsilon>0$, there exists a random walk satisfying
$\mathbf{E}\vert X\vert^{d-2-\varepsilon}<\infty$ and 
\begin{equation*}
     \limsup_{\vert z\vert\to\infty}\vert z\vert^{d-2}G_{{\bf R}^d}(0,z)=\infty.
\end{equation*}
Uchiyama considers dimensions $4$ and $5$ only, but it is quite simple
to show that this statement holds in every dimension $d\ge5$. 
We now give an example in our setting of a random walk which shows the optimality of 
Uchiyama's condition and of the moment condition in Theorem~\ref{thm:deep}.
Our example is just a multidimensional variation
of the classical Williamson example, see \cite{Williamson68}.

Let $d$ be greater than $4$ and consider $X$ with the following distribution.
For every $n\ge1$ and for every basis vector $e_k$ put 
\begin{equation*}
\mathbf{P}(X=\pm 2^n e_k)=\frac{q_n}{2d},
\end{equation*}
where the sequence $q_n$ is such that 
\begin{equation*}
\sum_{n=1}^\infty q_n=1\quad\text{and}\quad
q_n\sim \frac{c\log n}{2^{n(d-2)}}.
\end{equation*}
Clearly,
\begin{equation*}
\mathbf{E}\vert X\vert^{d-2}=\infty\quad\text{and}\quad
\mathbf{E}\frac{\vert X\vert^{d-2}}{\log^{1+\varepsilon}\vert X\vert}<\infty.
\end{equation*}
Using now the obvious inequality $G_{{\bf R}^d}(0,x)\ge\mathbf{P}(X=x)$,
we conclude that for every $j=1,\ldots,d$,
\begin{equation*}
\lim_{n\to\infty}2^{(d-2)n}G_{{\bf R}^d}(0,\pm 2^n e_j)=\infty.
\end{equation*}

If we have a cone $K$ such that $p\ge2$ and $e_j\in\Sigma$ for some $j$, then,
choosing $q_n\sim\frac{c\log n}{2^{n(p+d-2)}}$, we also have
\begin{equation*}
\lim_{n\to\infty}2^{(p+d-2)n}G_{K}(e_j,(1+2^n)e_j)=\infty.
\end{equation*}
Therefore, the finiteness of $\e\vert X(1)\vert^{r_1(p)}$ cannot be replaced by a weaker moment assumption.

But Uchiyama shows that the moment assumption $\mathbf{E}\vert X\vert^{d-2}$ is not necessary, as it can be replaced by $\mathbf{P}(X=x)=o(\vert x\vert^{-d-2})$, which implies the existence of the second moment
only. In Theorem~\ref{thm:deep} we have a similar situation: the moment condition $\mathbf{E}\vert X\vert^{r_1(p)}<\infty$
is not necessary and can be replaced by the assumption \ref{local_assump}, which yields the finiteness
of $\mathbf{E}\vert X\vert^{p\vee2}$ only. It has been shown in \cite{DeWa-15} that if $p>2$, the condition
$\mathbf{E}\vert X\vert^{p}<\infty$ is an optimal moment condition for the existence of the harmonic function
$V(x)$. 

Clearly, one can adapt the random walk from the example above to show that the moment 
assumption in the second statement of Theorem~\ref{thm:asymp_Green_boundary} is minimal as well. Indeed, it suffices to take 
$q_n\sim\frac{c\log n}{2^{n(p+d-1)}}$ and to assume that one of the vectors $\pm e_j$ belongs to the 
boundary of the cone $K$.

In order to show that the moment conditions
in the first claim of Theorem~\ref{thm:asymp_Green_boundary} are nearly minimal we consider the cone $K={\bf R}_+^d$, $d\ge3$. Clearly, $p=d$ for this cone. Set $\sigma=(1,0,0,\ldots,0)$. Then one has $K_\sigma={\bf R}\times{\bf R}_+^{d-1}$ and $q_\sigma=d-1$. We assume again that
\begin{equation*}
\mathbf{P}(X=\pm 2^ne_k)=\frac{q_n}{2d}.
\end{equation*}
This time we choose $q_n\sim c\frac{\log n}{2^{3(d-1)n}}$. Denoting by $\mathbf{1}$ the vector $e_1+\cdots +e_d$, we obtain that as $n\to\infty$,
\begin{equation*}
     G_K(\mathbf{1},\mathbf{1}+2^ne_1)\gg 2^{-3(d-1)n}. 
\end{equation*}
Moreover, it is rather simple to see that $\e[u(y_{\rho}),\tau'_{y}>\theta_{y}]$ converges to a positive constant for $\mathbf{1}+2^ne_1$. As a result, the first statement may fail for a random walk with $\mathbf{E}\vert X\vert^{3d-3}=\infty$. Remark that the first statement requires not only finiteness of moment of order $p+q_{\sigma}+d-2+(2-p)^+$, but also finiteness of some moment strictly greater than $p+q_{\sigma}+d-2+(2-p)^+$. We conjecture that this condition is actually sharp when $d\geq 3$.

\section{Boundary asymptotics of the survival probability}
\label{sec:survival}

The goal of this section is to collect lower bounds on the survival probability at time $n\geq 1$ of the random walk starting at $x$ when $n=o(\vert x\vert^{2})$ and $x\rightarrow \infty$ while $\frac{x}{\vert x\vert}\rightarrow \sigma\in \partial K$. Those bounds are used in the proof of our main results. The strategy of the proof is to compare the tangent cone at $\sigma$ with some smaller cones included in $K$. Let us give a first recall a useful result from \cite[Lem.~21]{RaTa-18}. 
\begin{lemma}\label{Lemma_24_revisited}
Let $0\leq r\leq p$ and $A>0$, and suppose that the increment $X$ admits moments of order $\kappa>r$. Set 
\begin{equation*}
     S(x,n)^{+}=\sup_{1\leq \ell\leq n^{1-\epsilon}}
\vert S(\ell)\vert \mathbf{1}_{\tau_{x}>\ell}.
\end{equation*}
Then, for each $s<(\kappa-r)/2$ and $\beta\in ((p/2-1)\wedge 0,p/2)$, there exists $C>0$ such that
\begin{equation*}
     \e\bigl[(S(x,n)^{+})^{r};S(x,n)^{+}\geq n^{1/2-\epsilon/8}\bigr]\leq Cn^{-s}n^{1-(p/2-\beta)}(1+\vert x\vert)^{p-2\beta}\end{equation*}
for all $x\in K$. In particular, uniformly on $x\in K$, $\vert x\vert\leq A\sqrt{n}$,
\begin{equation*}
     \e\bigl[(S(x,n)^{+})^{r};S(x,n)^{+}\geq n^{1/2-\epsilon/8}\bigr]\leq Cn^{-s+1}.
\end{equation*}
\end{lemma} 

Recall from \ref{H:strongly_irreducible} that a random walk $S$ is strongly irreducible in a cone $K$ if there exists a constant $R>\text{diam }\Lambda$ such that for any $z\in C\cap \Lambda$, there exists a path with positive probability in $K\cap B(z,R)$ which starts in $z+K$ and ends at $z$. If $K$ is a cone with exponent $q$ such that $S$ is strongly irreducible in $K$, then there exists $c>0$ such that for all $z\in{K}$ and all $n\geq 1$,
\begin{equation}\label{lowerBoundProbSurv}
\pr(\tau_{z}>n)\geq cn^{-q/2}.
\end{equation}
See \cite[Lem.~13]{RaTa-18} for a proof of this fact.

We now prove that a tangent cone can be well approximated by a smaller cone included in the original cone. We recall that $K_{\sigma}$ denotes the tangent cone to $K$ at $\sigma$, see \eqref{eq:def_K_sigma}, and for $\alpha>0$ we set 
\begin{equation*}
     K_{\sigma,\alpha}=\lbrace x\in K_{\sigma}: \alpha\vert x-\sigma\vert\leq \dist(x,\partial K_{\sigma})\rbrace.
\end{equation*}
Notice that for $\alpha$ small enough, $K_{\sigma,\alpha}$ is a non-empty cone. For $\epsilon>0$, let
\begin{equation*}
     V_{\epsilon}(\sigma)=B(\sigma,\epsilon)\cap K\quad \text{and}\quad \partial V_{\epsilon}(\sigma)=B(\sigma,\epsilon)\cap \partial K.
\end{equation*}
Hereafter, $(z-\sigma)+K_{\sigma,\alpha}$ denotes the translated version of $K_{\sigma,\alpha}$ with origin at $z$.
\begin{lemma}\label{inclusion_smaller_cone}
For all $\alpha>0$ sufficiently small, there exist $\epsilon,\alpha'>0$ such that for all $y\in \partial V_{\epsilon}(\sigma)$ and all $z\in (y-\sigma)+K_{\sigma,\alpha}\cap B(\sigma,\epsilon)$, one has $z\in K$ and
\begin{equation*}\dist(z,\partial K)\geq\alpha'\vert z-y\vert.\end{equation*}
\end{lemma}
The proof of the above lemma uses a few basic facts from convex analysis. Recall that for a convex function $\phi: C\rightarrow {\bf R}$ defined on an open convex set $C\subset {\bf R}^{d-1}$, we define the subgradient $\partial \phi (x)$ of $\phi$ at $x\in C$ by 
\begin{equation*}
     \partial \phi (x)=\lbrace v\in {\bf R}^{d-1}: \forall u\in C,\,\phi(u)-\phi(x)\geq \langle v,u-x\rangle\rbrace.
\end{equation*}
The subgradient is upper-semicontinuous in the following sense: if $x_{n}\rightarrow x$ and $v_{n}\rightarrow v$ with $v_{n}\in \partial\phi(x_{n})$ for any $n$, then $v\in \partial\phi(x)$. 

For $s\in{\bf R}^{d-1}$, the convex function $\phi$ admits a directional derivative $\phi_{s}(x)$ at any point $x\in C$, and we have
\begin{equation*}
     \phi_{s}(x)=\max_{v\in \partial \phi(x)}\langle v,s\rangle.
\end{equation*}
Note that the upper-semicontinuity of the subgradient implies a uniform upper-semicontinuity of the directional derivatives.
\begin{lemma}\label{upper_continuity_derivative}
Let $x\in C$ and $\epsilon>0$. There exists a neighborhood $V$ of $x$ such that 
\begin{equation*}
     \phi_{s}(u)\leq \phi_{s}(x)+\epsilon
\end{equation*}
for all $u\in V$ and $s\in {\bf S}^{d-2}$.
\end{lemma}
\begin{proof}
Let us prove the statement by contradiction. Assume the existence of a sequence $(x_{n},s_{n})$ in $C\times {\bf S}^{d-2}$ such that $x_{n}\rightarrow x$ and $\phi_{s_{n}}(x_{n})> \phi_{s_{n}}(x)+\epsilon$. Up to taking a subsequence, we can assume that $s_{n}\rightarrow s\in {\bf S}^{d-2}$. For each $n$, let $v_{n}$ be the maximizer of $\langle v,s_{n}\rangle$ for $v\in \partial\phi_{x_{n}}$. Since $\partial \phi$ is uniformly bounded on a neighborhood of $x$, we can assume by compactness that $v_{n}$ converges to a vector $v_{0}$. By upper-semicontinuity of $\partial \phi$, one has $v_{0}\in \partial \phi(x)$. Then we have
\begin{equation}\label{contradiction}
     \phi_{s_{n}}(x_{n})=\langle v_{n},s_{n}\rangle \rightarrow \langle v_{0},s\rangle\leq \max_{v\in\partial \phi(x)}\langle v,s\rangle\leq \phi_{s}(x),
\end{equation}
and on the other hand
\begin{equation*}
     \langle v_{n},s_{n}\rangle= \phi_{s_{n}}(x_{n})\geq \phi_{s_{n}}(x)+\epsilon.
\end{equation*}
Since $s\mapsto \phi_{s}(x)$ is continuous and $s_{n}\rightarrow s$, $\phi_{s_{n}}(x_{n})\geq \phi_{s}(x) +\epsilon/2$ for $n$ large enough, and by \eqref{contradiction} we get a contradiction.
\end{proof}
\begin{proof}[Proof of Lemma \ref{inclusion_smaller_cone}]
Up to an isometry of ${\bf R}^{d}$, we can assume $\sigma=0$ and that $(0,\ldots,0,1)$ is a vector pointing inside $K$. Let $V$ be a neighborhood of $0$ in $H_{d}:=\lbrace x\in {\bf R}^{d}:x_{d}=0\rbrace$ such that there exists a convex function $\phi : V\rightarrow {\bf R}$ with Lipschitz constant $M$ whose graph is locally the boundary of $K$ around $\sigma$. We further assume that there exists $\epsilon>0$ such that
\begin{equation*}
     \lbrace (y,t)\in V\times {\bf R}:\phi(y)<t<\phi(y)+\epsilon\rbrace \subset K.
\end{equation*}
Such $\epsilon$ always exists if we assume $V$ small enough.

Note that the tangent cone of $K$ at $\sigma$ is exactly the set 
\begin{equation*}
     K_{\sigma}=\lbrace (y,x_{d})\in {\bf R}^{d-1}\times {\bf R}: x_{d}\geq  \phi_{y}(0)\rbrace.
\end{equation*}
Let $\alpha$ be small enough so that $K_{\sigma,\alpha}$ is non-empty. For $\beta>0$, set
\begin{equation*}
     \tilde{K}_{\beta}:=\lbrace  (y,x_{d})\in {\bf R}^{d-1}\times {\bf R}: x_{d}\geq  \phi_{y}(0)+\beta\vert y\vert\rbrace.\end{equation*}
Then, $(\tilde{K}_{\beta})_{\beta>0}$ is a decreasing sequence of cones and $\bigcup_{\beta>0} \tilde{K}_{\beta}=K_{\sigma}$, hence there exists $\alpha'>0$ such that $K_{\sigma,\alpha}\subset \tilde{K}_{\alpha'}$. By Lemma \ref{upper_continuity_derivative}, let $\epsilon'<\epsilon$ be such that $B_{{\bf R}^{d-1}}(0,\epsilon')\subset V$ is a neighborhood of $0$, with the property that for each $y\in B_{{\bf R}^{d-1}}(0,\epsilon')$ and $s\in{\bf S}^{d-2}$, we have 
\begin{equation}
\label{bound_directional}
     \phi_{s}(y)\leq \phi_{s}(0)+\alpha'/2.
\end{equation}
Let $z\in y+\tilde{K}_{\alpha'}\cap B(y,\epsilon'/2)$ with $y=(y_{1},\phi(y_{1}))\in\partial K$ and $y_{1}\in B_{{\bf R}^{d-1}}(0,\epsilon'/2)$. Writing $z=(z_{1},z_{2})\in {\bf R}^{d-1}\times {\bf R}$, we have on the first hand 
\begin{equation*}
     z_{2}-\phi(y_{1})\geq \phi_{z_{1}-y_{1}}(0)+\alpha'\vert z_{1}-y_{1}\vert.
\end{equation*}
On the other hand, integrating \eqref{bound_directional} on the segment $[y_{1},z_{1}]\subset B_{{\bf R}^{d-1}}(0,\epsilon')$ yields
\begin{equation*}
     \phi(z_{1})-\phi(y_{1})=\int_{0}^{1}\phi_{z_{1}-y_{1}}(y_{1}+t(z_{1}-y_{1}))dt
\leq \phi_{z_{1}-y_{1}}(0)+\alpha'/2\vert z_{1}-y_{1}\vert\leq z_{2}-\phi(y_{1}).
\end{equation*}
Hence, $z_{2}\geq \phi(z_{1})+\alpha'/2\vert z_{1}-y_{1}\vert$, which yields 
\begin{equation}\label{minimum_slope}
\phi(z_{1})+\alpha'/2\vert z_{1}-y_{1}\vert<z_{2}<\phi(z_{1})+\epsilon/2
\end{equation}
by the choice of $\epsilon'$. Since $z_{1}\in V$, $(z_{1},u)\in K$ for all $u\in (\phi(z_{1}),\phi(z_{1})+\epsilon)$, which implies that $(z_{1},z_{2})\in K$. Therefore, for $y\in \partial V_{\epsilon'/2}(\sigma)$ we have $(y-\sigma)+\tilde{K}_{\alpha'}\cap B(\sigma,\epsilon'/2)\subset K$. Since $K_{\sigma,\alpha}\subset\tilde{K}_{\alpha'}$, we also have $(y-\sigma)+K_{\sigma,\alpha}\cap B(\sigma,\epsilon'/2)\subset K$ for all $y\in \partial V_{\epsilon'/2}(\sigma)$. 

Since $\phi$ is Lipschitz with Lipschitz constant $M>0$ on  $B_{{\bf R}^{d-1}}(0,\epsilon')$, standard geometric arguments yield that for $c=\sin(\arctan(1/M))$,
\begin{equation*}d(z,\partial K)\geq c(z_{2}-\phi(z_{1})),\end{equation*}
when  $z=(z_{1},z_{2})\in K$ with $z_{1}\in B_{{\bf R}^{d-1}}(0,\epsilon'/2)$ and $z_{2}\leq \epsilon'/2$. Thus, \eqref{minimum_slope} yields that 
\begin{equation}\label{first_coordinate}
d(z,\partial K)\geq  \tfrac{c\alpha'}{2} \vert z_{1}-y_{1}\vert.
\end{equation}
Since the Lipschitz property also yields $\vert z_{1}-y_{1}\vert\geq \vert \phi(z_{1})-\phi(y_{1})\vert/M$, we deduce that
\begin{equation*}
     d(z,\partial K)\geq \tfrac{c\alpha'}{2M}\vert \phi(z_{1})-\phi(y_{1})\vert.
\end{equation*}
Hence, for $c'=\min\{c,\tfrac{c\alpha'}{2M}\}$, 
\begin{equation}\label{second_coordinate}
d(z,\partial K)\geq \tfrac{c'}{2}(z_{2}-\phi(z_{1})+\phi(z_{1})-\phi(y_{1}))\geq  \tfrac{c'}{2}\vert z_{2}-\phi(y_{1})\vert.
\end{equation}
Let $t$ be such that $\vert y-z\vert \leq t\max\{\vert y_{1}-z_{1}\vert,\vert z_{2}-y_{2}\vert\}$. Then, since $y_{2}=\phi(y_{1})$,
\begin{equation*}
     d(z,\partial K)\geq  \tfrac{c'}{2t}\vert y-z\vert.
\end{equation*}
This concludes the proof of the second statement.
\end{proof}

\begin{proposition}
\label{lower_bound_survival_tangent}
Suppose \ref{H:strongly_irreducible} that $S$ is strongly irreducible in $K$. Let $\sigma\in\partial K$ and $q_{\sigma}$ the exponent associated to the corresponding tangent cone $K_{\sigma}$. Then, for all $q'>q_{\sigma}$ and $\epsilon>0$ small enough, there exists $c>0$ such that for all $x$ large enough with $\frac{x}{\vert x\vert} \rightarrow \sigma$ and for all $n\leq \vert x\vert^{2-\epsilon}$,
\begin{equation*}\pr(\tau_{x}>n)>cn^{-q'/2}.\end{equation*}
\end{proposition}
\begin{proof}
Let $q'>q_{\sigma}$ be small enough, and let $\alpha>0$ be such that $q_{K_{\sigma,\alpha}}=q'$. Such $\alpha$ exists, since $K_{\sigma,\alpha}\cap {\bf S}^{d-1}$ converges in Hausdorff distance to $K_{\sigma}$ as $\alpha$ goes to zero. Similarly to the proof of Lemma \ref{inclusion_smaller_cone}, assume without loss of generality that $K\subset {\bf R}^{d-1}\times {\bf R}^{+}$, $\sigma=(1,0,\ldots,0)$ and $v=(0,\ldots,0,1)$ is a vector pointing towards the interior of $K$. For $x\in K$, let $x_{\sslash}$ be the projection of $x$ on $\partial K$ along $(0,\ldots ,0,1)$. As $x$ goes to infinity while $x/\vert x\vert\rightarrow \sigma$, $x_{\sslash}/\vert x\vert$ converges to $\sigma$ and $\vert x/\vert x\vert-x_{\sslash}\vert\rightarrow 0$. 
 
By Lemma \ref{inclusion_smaller_cone}, there exist $\eta,\alpha'>0$ such that for all $z\in \partial V_{\eta}(\sigma)$ and all $u\in (z-\sigma)+K_{\sigma,\alpha}\cap B(z,\eta)$, $\dist(u,\partial K)\geq \alpha'\vert u-z\vert$. For $\alpha$ small enough and $\vert x\vert$ large enough, $x/\vert x\vert\in (x_{\sslash}/\vert x\vert-\sigma)+K_{\sigma,\alpha}$, with $x_{\sslash}/\vert x\vert\in \partial V_{\eta}(\sigma)$, which yields then that
\begin{equation}\label{linear_distance_boundary}
\dist(x,\partial K)\geq \alpha'\vert x-x_{\sslash}\vert.
\end{equation}
For $\alpha$ small enough so that $v+\sigma$ points towards the interior of $K_{\sigma,\alpha}$, let $t>0$ be  such that the harmonic function $V_{K_{\sigma,\alpha}}(\sigma+tv)$ is positive. The existence of $t$ is guaranteed by \cite[Thm~1]{DeWa-15}, which gives also $c>0$ such that $\pr(\tau_{tv,K_{\sigma,\alpha}-\sigma}>n)\geq cn^{-q'/2}$ for all $n\geq 1$. Hence, for $x$ such that $\vert x-x_{\sslash}\vert>t$, $x-tv\in x_{\sslash}-\sigma+K_{\sigma,\alpha}$ and 
\begin{equation}\label{lower_bound_smaller_cone}
\pr(\tau_{x,x_{\sslash}-\sigma+K_{\sigma,\alpha}}>n)\geq \pr(\tau_{x,x-tv+K_{\sigma,\alpha}}>n)\geq cn^{-q'/2}.
\end{equation}
Let us assume from now on that $\vert x-x_{\sslash}\vert \geq t$. Suppose first that $n\geq \vert x-x_{\sslash}\vert^{2-\epsilon}$. Thanks to the moments assumption \ref{H:moments}, we can apply the first part of Lemma \ref{Lemma_24_revisited} to the random walk in $ K':=x_{\sslash}-\sigma+K_{\sigma,\alpha}$ with $r=0$, $\kappa>q'+2$ small enough and $\epsilon'$ small enough  to get
\begin{equation*}\pr(\sup_{1\leq l\leq n}\vert S(l)\vert\geq n^{1/2+\epsilon'}, \tau_{x,K'}\geq n)\leq Cn^{-s},\end{equation*}
with $s>q'/2$, for $n\geq \vert x-x_{\sslash}\vert^{2-\epsilon}$. Hence, the latter inequality together with \eqref{lower_bound_smaller_cone} yields
\begin{equation*}\pr(\sup_{1\leq l\leq n}\vert S(l)\vert\leq n^{1/2+\epsilon'}, \tau_{x,K'}\geq n)\geq cn^{-q'/2},\end{equation*}
for $n\geq \vert x-x_{\sslash}\vert^{2-\epsilon}$ and some $c>0$. Since $n\leq \vert x\vert^{2-\epsilon}$, choosing $\epsilon'$ small enough implies that
\begin{equation*}\pr(\sup_{1\leq l\leq n}\vert S(l)\vert\leq \vert x\vert^{1-\epsilon''}, \tau_{x,K'}\geq n)\geq cn^{-q'}.\end{equation*}
Since, by Lemma \ref{inclusion_smaller_cone}, $(x_{\sslash}-\sigma)+K_{\sigma,\alpha}\cap B(x_{\sslash},\eta \vert x\vert)\subset K$, the latter inequality implies that 
\begin{equation*}\pr(\tau_{x,K}\geq n)\geq \pr(\sup_{1\leq l\leq n}\vert S(l)\vert\leq \vert x\vert^{1-\epsilon''}, \tau_{x,K'}\geq n)\geq c'n^{-q'}.\end{equation*}
for all $n\geq \vert x-x_{\sslash}\vert^{2-\epsilon}$ and $x$ large enough. 

Suppose now that $n\leq \vert x-x_{\sslash}\vert^{2-\epsilon}$. Applying Doob and Rosenthal inequalities together with \eqref{linear_distance_boundary} gives
\begin{align*}
\pr(\tau_{x}\leq n)&\leq \pr(\sup_{1\leq k\leq n} \vert S_{k}\vert\geq \dist(x,\partial K))\\
&\leq \pr(\sup_{1\leq k\leq n} \vert S_{k}\vert\geq \alpha'\vert x-x_{\sslash}\vert)\\
&\leq \frac{2n\e[\vert X\vert^{2}]}{\alpha'^{2}\vert x-x_{\sslash}\vert^{2}}\\
&\leq  Cn^{-\epsilon/(2-\epsilon)}.
\end{align*}
Hence, there exist $c,N>0$ such that $\pr(\tau_{x}> n)\geq c$ for $n>N$ with $n\leq \vert x-x_{\sslash}\vert^{2-\epsilon}$.

Suppose finally that $\vert x-x_{\sslash}\vert\leq t$. By the proof of \cite[Lem.~14]{DeWa-15} and the strong irreducibility of $S$ in $K$, there exist $c,\rho,n_{0}>0$ such that for $x$ large enough, we have
\begin{equation*}
     \pr(\vert x+S(n_{0})-(x+S(n_{0}))_{\sslash}\vert\geq t, \vert S(n_{0})\vert\leq n_{0}R)\geq \rho.
\end{equation*}
Hence, for $n\geq n_{0}$, by the Markov property and the first part of the proof,
\begin{align*}
     \pr(\tau_{x}>n)&\geq \e[\tau_{x+S(n_{0})}\geq n-n_{0};\vert x+S(n_{0})-(x+S(n_{0}))_{\sslash}\vert\geq t, \vert S(n_{0})\vert\leq n_{0}R]\\
&\geq  c\rho (n-n_{0})^{-q'}.
\end{align*}
This gives the result for $n$ large enough.
\end{proof}
We also give an asymptotic lower bound of the réduite $u$ along the boundary, which are sharper than \eqref{lowerBoundu}.

\begin{lemma}
\label{lower_bound_u_tangent}
Let $\sigma\in \partial \Sigma$ and $q''>q'>q_{\sigma}$. Then there exists $c>0$ such that uniformly on $x$ going to infinity while $x/\vert x\vert\rightarrow \sigma$ and $\dist(x,\partial K)=o(\vert x\vert)$,
\begin{equation*}u(x)\geq c\vert x\vert^{p-q''}\dist(x,\partial K)^{q'}.\end{equation*}
\end{lemma}
\begin{proof}
We use the same notations as in the previous proof and take $\alpha>0$ such that $q_{K_{\sigma,\alpha}}=q'$. Let $x$ going to infinity with $x/\vert x\vert\rightarrow \sigma$. By Lemma \ref{inclusion_smaller_cone}, there exists $\epsilon>0$ such that for $t>0$ and $x$ large enough,
\begin{equation*}
     \pr(\tau_{x}^{\bm}\geq t)\geq \pr(\tau>t, \sup_{0\leq u\leq t}\vert B(u)\vert\leq \epsilon \vert x-x_{\sslash}\vert),
\end{equation*}
where we have set $\tau:=\tau_{x,x_{\sslash}-\sigma+K_{\sigma,\alpha}}^{\bm}$. 

Let us show that $\pr(\tau>t, \sup_{0\leq u\leq t}\vert B(u)\vert> \epsilon \vert x-x_{\sslash}\vert)$ is negligible in comparison with $\pr(\tau>t)$, by adapting the reflection principle to a Brownian motion in a cone. By conditioning on the last time $\theta$ when $B$ reaches the sphere of radius $\epsilon\vert x -x_{\sslash}\vert$, we get
\begin{equation*}
     \pr(\tau>t, \sup_{0\leq u\leq t}\vert B(u)\vert> \epsilon \vert x-x_{\sslash}\vert, \vert B(t)\vert \leq \epsilon \vert x-x_{\sslash}\vert)\leq \pr(\theta<t<\tau, \langle B_{t}-B_{\theta},B_{\theta}\rangle<0).
\end{equation*}
We denote by $B^{(\theta)}$ the process $B_{\theta+u}-B_{\theta}$, which is a Brownian meander independent of $(B(u))_{0\leq u\leq \theta}$, see for example \cite{RY}.
Denote by $K^{+}$ (resp.\ $K^{-}$) the intersection of $K$ with the set
\begin{equation*}
     \lbrace v\in {\bf R}^{d}:\langle v,B_{\theta}\rangle >\langle B_{\theta},B_{\theta}\rangle\rbrace\quad (\text{resp.} \quad
     \lbrace v\in {\bf R}^{d}:\langle v,B_{\theta}\rangle <\langle B_{\theta},B_{\theta}\rangle\rbrace),
\end{equation*}
and let $s$ denote the symmetry with respect to the hyperplane $B_{\theta}+B_{\theta}^{\perp}$. Since we have $s(K^{-})\subset K^{+}$, 
\begin{equation*}
     (B(\theta)+B^{\theta}_{u})_{0\leq u\leq t-\theta}\subset K^{-}\quad \text{implies that}\quad s\left((B(\theta)+B^{\theta}_{u})_{0\leq u\leq t-\theta}\right)\subset K^{+}.
\end{equation*}
Moreover, $s$ turns a negative meander into a positive one, and is thus measure preserving. Therefore,
\begin{equation*}
     \pr(s(\lbrace (B(\theta)+B^{\theta}_{u})_{0\leq u\leq t-\theta}\subset K^{-}\rbrace))=\pr(\lbrace (B(\theta)+B^{\theta}_{u})_{0\leq u\leq t-\theta}\subset K^{-}\rbrace).
\end{equation*}
This implies that 
\begin{equation*}
     \pr(\theta<t<\tau, \langle B_{t}-B_{\theta},B_{\theta}\rangle<0)\leq \pr(\theta<t<\tau, \langle B_{t}-B_{\theta},B_{\theta}\rangle>0).
\end{equation*}
Since $\langle B_{t}-B_{\theta},B_{\theta}\rangle>0$ implies that $\vert B_{t}\vert>\vert B_{\theta}\vert$, we get finally
\begin{multline*}
\pr\left(\tau>t, \sup_{0\leq u\leq t}\vert B(u)\vert> \epsilon \vert x-x_{\sslash}\vert, \vert B(t)\vert \leq \epsilon \vert x-x_{\sslash}\vert\right)\\
\leq \pr\left(\tau>t, \sup_{0\leq u\leq t}\vert B(u)\vert> \epsilon \vert x-x_{\sslash}\vert, \vert B(t)\vert > \epsilon \vert x-x_{\sslash}\vert\right).
\end{multline*}
Since 
\begin{equation*}
     \pr(\tau>t, \sup_{0\leq u\leq t}\vert B(u)\vert> \epsilon \vert x-x_{\sslash}\vert, \vert B(t)\vert > \epsilon \vert x-x_{\sslash}\vert)= \pr(\tau>t, \vert B(t)\vert > \epsilon \vert x-x_{\sslash}\vert),
\end{equation*}
we have
\begin{align*}
\pr(\tau>t, \sup_{0\leq u\leq t}\vert B(u)\vert> \epsilon \vert x-x_{\sslash}\vert)&=\pr(\tau>t, \sup_{0\leq u\leq t}\vert B(u)\vert> \epsilon \vert x-x_{\sslash}\vert, \vert B(t)\vert \leq \epsilon \vert x-x_{\sslash}\vert)\\
&+\pr(\tau>t, \sup_{0\leq u\leq t}\vert B(u)\vert> \epsilon \vert x-x_{\sslash}\vert, \vert B(t)\vert > \epsilon \vert x-x_{\sslash}\vert)\\
&\leq 2\pr(\tau>t, \sup_{0\leq u\leq t}\vert B(u)\vert> \epsilon \vert x-x_{\sslash}\vert, \vert B(t)\vert > \epsilon \vert x-x_{\sslash}\vert)\\
&\leq 2\pr(\tau>t,\vert B(t)\vert > \epsilon \vert x-x_{\sslash}\vert).
\end{align*}
Therefore, using \cite[Lem.~18]{DeWa-15} yields for $t=o(\vert x-x_{\sslash}\vert^{2})$,
\begin{equation*}
     \pr(\tau>t, \sup_{0\leq u\leq t}\vert B(u)\vert> \epsilon \vert x-x_{\sslash}\vert)=o(\pr(\tau>t)),
\end{equation*}
and finally
\begin{equation*}
     \pr(\tau>t, \sup_{0\leq u\leq t}\vert B(u)\vert\leq \epsilon \vert x-x_{\sslash}\vert)\sim \frac{u_{K_{\sigma,\alpha}}(x-x_{\sslash})}{t^{q'/2}},
\end{equation*}
uniformly for $t$ and $x$ such that $x-x_{\sslash}=o(\sqrt{t})$. Since by Lemma \ref{inclusion_smaller_cone},
\begin{equation*}
     \pr(\tau>t, \sup_{0\leq u\leq t}\vert B(u)\vert\leq \epsilon \vert x-x_{\sslash}\vert)\leq \pr(\tau^{\bm}_{x}>t),
\end{equation*}
and $u_{K_{\sigma,\alpha}}(x-x_{\sslash})\geq c \dist(x-x_{\sslash},\partial K_{\sigma,\alpha})^{q'}$ by \eqref{lowerBoundu},
we have 
\begin{equation*}
     \pr(\tau^{\bm}_{x}>t)\geq \frac{c \dist(x-x_{\sslash},\partial K_{\sigma,\alpha})^{q'}}{t^{q'/2}}.
\end{equation*}
Usual Gaussian estimates in $K$ (see for example \cite[App.~A]{RaTa-18}) yields therefore
\begin{equation*}
     \frac{u(x)}{t^{p/2}}\geq c\pr(\tau_{x}^{\bm}>t)\geq \frac{c\dist(x-x_{\sslash},\partial K_{\sigma,\alpha})^{q'}}{t^{q'/2}}
\end{equation*}
for $x$ going to infinity with $x/\vert x\vert\rightarrow \sigma$, $\dist(x-x_{\sslash},\partial K_{\sigma,\alpha})=o(\vert x \vert)$ and $x-x_{\sslash}=o(\sqrt{t})$. Hence, evaluating the above inequality at $t=\vert x\vert ^{2(p-q'')/(p-q')}$ for $q''>q'$ small enough gives for any $q'>q$ the existence of $c>0$ such that
\begin{equation*}
     u(x)\geq c \vert x\vert^{p-q''}\dist(x-x_{\sslash},\partial K_{\sigma,\alpha})^{q'}.\qedhere
\end{equation*}
\end{proof}

\end{document}